\documentclass{article}
\usepackage{amsmath,amssymb,amsthm,color,graphicx,diagbox,tikz,colortbl,array}

\title{Finite Kripke models and provability interpretations in quantified modal logic}
\author{Haruka Kogure\footnote{Email: kogure@stu.kobe-u.ac.jp}
\footnote{Graduate School of System Informatics, Kobe University, 1-1 Rokkodai, Nada, Kobe 657-8501, Japan.}
and Taishi Kurahashi\footnote{Email: kurahashi@people.kobe-u.ac.jp}
\footnote{Graduate School of System Informatics, Kobe University, 1-1 Rokkodai, Nada, Kobe 657-8501, Japan.}}
\date{}

\theoremstyle{plain}
\newtheorem{theorem}{Theorem}[section]
\newtheorem{lemma}[theorem]{Lemma}
\newtheorem{proposition}[theorem]{Proposition}
\newtheorem{corollary}[theorem]{Corollary}

\newtheorem{problem}[theorem]{Problem}
\newtheorem{claim}{Claim}

\theoremstyle{definition}
\newtheorem{definition}[theorem]{Definition}

\newcommand{\PA}{\mathsf{PA}}
\newcommand{\PR}{\mathrm{Pr}}

\newcommand{\Prov}{\mathrm{Prov}}
\newcommand{\Proof}{\mathrm{Proof}}
\newcommand{\Con}{\mathrm{Con}}

\newcommand{\gn}[1]{\ulcorner#1\urcorner}

\newcommand{\DC}[1]{\mathbf{D#1}}
\newcommand{\DCU}[1]{\mathbf{D#1^{U}}}
\newcommand{\DCG}[1]{\mathbf{D#1^{G}}}

\newcommand{\SCU}{\mathbf{\Sigma_1 C^{U}}}

\newcommand{\num}{\overline}
\newcommand{\LA}{\mathcal{L}_A}
\newcommand{\su}{\mathsf{s}}
\newcommand{\ora}{\overrightarrow}

\newcommand{\GL}{\mathbf{GL}}
\newcommand{\QK}{\mathbf{QK}}
\newcommand{\QKbf}{\mathbf{QKbf}}
\newcommand{\QGL}{\mathbf{QGL}}
\newcommand{\QPL}{\mathbf{QPL}}
\newcommand{\lh}{\mathrm{lh}}

\begin{document}

\maketitle

\begin{abstract}
In this paper, we investigate arithmetical completeness with respect to finite Kripke models of quantified modal logic. 
We adapt the finite-model embedding techniques of Artemov and Japaridze to two settings involving finite Kripke models.
First, for conversely well-founded finite Kripke models of quantified modal logic, we construct a $\Sigma_2$ Fefermanian provability predicate together with an arithmetical interpretation that embeds the model into arithmetic.
Second, for finite constant domain Kripke models of quantified modal logic, we construct a $\Sigma_1$ provability predicate satisfying $\mathbf{D2^G}$ and an arithmetical interpretation yielding such an embedding.
\end{abstract}

\section{Introduction}\label{sec:intro}

Solovay's arithmetical completeness theorem \cite{Solo76}, which states that the modal propositional logic $\GL$ coincides exactly with the provability logic of $\PA$, has become a fundamental cornerstone in the study of provability predicates via modal logic. 
In particular, the method of embedding transitive and conversely well-founded finite Kripke frames into arithmetic by means of the so-called Solovay function remains a central and indispensable technique to this day (see the survey papers \cite{JD98, AB05} on provability logic for several applications of this method).  
In recent years, the authors have established arithmetical completeness results for $\mathbf{K}$ and for several non-normal modal logics. 
For instance, the arithmetical completeness theorem for the modal propositional logic $\mathbf{K}$ with respect to $\Sigma_2$ Fefermanian provability predicates was established in \cite{Kura18}. 
The key technique used in its proof is again an application of Solovay-style constructions. 
In addition, the arithmetical completeness theorem for $\mathbf{K}$ with respect to $\Sigma_1$ provability predicates was proved in \cite{Kura24}. 
An overview of recent studies on arithmetical completeness for several logics is given in \cite{KK}.

It is natural to attempt to extend this framework from propositional logic to predicate logic. 
However, it has been shown that many of the positive results obtained in the propositional setting do not extend to the predicate case. 
For example, Montagna \cite{Mont84} proved that the predicate extension $\QGL$ of $\GL$ is neither Kripke complete nor arithmetically complete. 
Artemov \cite{Arte85} proved that the quantified truth provability logic of $\PA$ is not arithmetical, and Vardanyan \cite{Vard86} proved that the quantified provability logic of $\PA$ is $\Pi^0_2$-complete (see also related results by Berarducci \cite{Bera89}, Boolos and McGee \cite{BM87}, and McGee \cite{McGe94}). 
Furthermore, Artemov \cite{Arte86} showed that the quantified provability logic of $\PA$ depends on the choice of formulas representing $\PA$ (see also \cite{Kura13,Kura22} for related results). 
These results indicate that, unlike in the propositional case, it is nontrivial to control quantified provability logic in a systematic way.

On the other hand, some positive results have been obtained. 
Among them, in this paper, we focus on the work of Artemov and Dzhaparidze, currently spelled Japaridze \cite{AD90}. 
They showed that, for finite Kripke models, Solovay's method can be extended from propositional logic to predicate logic. 
More precisely, they proved that for every transitive and conversely well-founded Kripke model of quantified modal logic with finitely many worlds and finite domains, and for every formula $A$ that is not valid in that model, there exists an arithmetical interpretation $f$ such that $\PA \nvdash f(A)$. 
Moreover, Japaridze \cite{Dzha90} further refined this result by providing more explicit sufficient conditions under which a Kripke model can be embedded into arithmetic.

The aim of the present paper is to combine the method of Artemov and Japaridze \cite{AD90} with the results on $\mathbf{K}$ obtained in \cite{Kura18,Kura24}. 
Our main results are as follows.

\begin{enumerate}
\item
For each conversely well-founded finite Kripke model of quantified modal logic, there exist a $\Sigma_2$ Fefermanian provability predicate of $\PA$ and a corresponding arithmetical interpretation such that the model can be embedded into arithmetic (Theorem \ref{main1}).

\item
For each finite constant domain Kripke model of quantified modal logic, there exist a $\Sigma_1$ provability predicate of $\PA$ satisfying the global version $\DCG{2}$ of the derivability condition $\DC{2}$ and a corresponding arithmetical interpretation such that the model can be embedded into arithmetic (Theorem \ref{main2}). 
\end{enumerate}

The first result can be obtained by a straightforward combination of the arguments of Artemov and Japaridze with those of \cite{Kura18}. 
Therefore, the technical contribution of this paper lies in the second result. 
Unlike the first result, it cannot be obtained merely by combining existing arguments. 
In order to connect quantified modal logic with arithmetic, one must overcome substantial difficulties arising from the treatment of quantifiers and substitutions of numerals. 
Our proof introduces essentially new ideas to address these issues. 

The paper is organized as follows.
In Section~\ref{sec:pre}, we introduce preliminaries on quantified modal logic, provability predicates, and arithmetical interpretations.
In Section~\ref{sec:main1}, we prove the first main theorem stating that conversely well-founded finite Kripke models can be embedded into arithmetic by using $\Sigma_2$ Fefermanian provability predicates.
In Section~\ref{sec:main2}, we establish our second main result stating that finite constant domain Kripke models can be embedded into arithmetic by means of $\Sigma_1$ provability predicates satisfying $\DCG{2}$.
Finally, in the Appendix, we prove the primitive recursiveness of the function $h$, which is used in the construction of the provability predicate in Section~\ref{sec:main2}.
The method developed in this appendix is expected to be useful in the construction of provability predicates whose behavior essentially involves free variables and substitutions.

\section{Preliminaries}\label{sec:pre}

This section consists of three subsections. 
In Subsection \ref{ssec:qml}, we recall some basics of quantified modal logic.
In Subsection \ref{ssec:pp}, we introduce formal arithmetic and provability predicates.
In Subsection \ref{ssec:ai}, we introduce arithmetical interpretations, which connect quantified modal logic and formal arithmetic.
In particular, we explain why we prove the two kinds of theorems mentioned in the introduction.

\subsection{Quantified modal logic}\label{ssec:qml}

The language of quantified modal logic consists of the following symbols:
\begin{itemize}
    \item countably many variables $x,y,z,\ldots$,
    \item for each $n \in \omega$, countably many $n$-ary predicate symbols $P,Q,\ldots$,
    \item logical symbols $\neg,\land,\lor,\to,\forall, \exists$,
    \item logical constants $\top$, $\bot$, 
    \item unary modal operator $\Box$.
\end{itemize}

The modal operator $\Diamond$ is defined by $\Diamond :\equiv \neg \Box \neg$.

\begin{definition}[The logic $\QK$]
The quantified modal logic $\QK$ is defined as follows.

\begin{itemize}
    \item Axioms:
    \begin{itemize}
        \item All axioms of first-order classical logic,
        \item The universal closures of all instances of $\Box(A \to B) \to (\Box A \to \Box B)$.
    \end{itemize}
    
    \item Rules:
    \begin{itemize}
        \item Modus Ponens: from $A$ and $A \to B$, infer $B$,
        \item Generalization: from $A$, infer $\forall x A$,
        \item Necessitation: from $A$, infer $\Box A$.
    \end{itemize}
\end{itemize}
\end{definition}

The quantified modal logic $\QGL$ is obtained from $\QK$ by adding the universal closures of all instances of $\Box(\Box A \to A) \to \Box A$ as axioms.

\begin{definition}[Kripke frames]
A \emph{Kripke frame} for quantified modal logic is a triple $\mathcal{F} = (W,R,\{D_w\}_{w \in W})$, where
\begin{itemize}
    \item $W$ is a nonempty set,
    \item $R$ is a binary relation on $W$,
    \item for each $w \in W$, $D_w$ is a nonempty set,
    \item if $x R y$, then $D_x \subseteq D_y$. 
\end{itemize}
We say a Kripke frame $\mathcal{F}$ is \emph{finite} if $W$ and all $D_w$ are finite sets. 
A Kripke frame $\mathcal{F}$ is said to be \emph{constant domain} if $D_x = D_y$ for all $x, y \in W$. 
When a Kripke frame is constant domain, we write $\mathcal{F} = (W,R,D)$.
\end{definition}

\begin{definition}[Kripke models]
A \emph{Kripke model} is a tuple $\mathcal{M} = (W, R, \{D_w\}_{w \in W}, \Vdash)$, where $\Vdash$ is a relation between elements of $W$ and atomic formulas with parameters from the corresponding domains, extended inductively as follows:
For each $w \in W$, 
\begin{itemize}
    \item $w \nVdash \bot$ and $w \Vdash \top$,
    \item $w \Vdash \neg A \iff w \nVdash A$,
    \item $w \Vdash A \land B \iff w \Vdash A \text{ and } w \Vdash B$,
    \item $w \Vdash A \lor B \iff w \Vdash A \text{ or } w \Vdash B$,
    \item $w \Vdash A \to B \iff w \nVdash A \text{ or } w \Vdash B$,
    \item $w \Vdash \forall x A \iff \forall a \in D_w \, (w \Vdash A(a))$,
    \item $w \Vdash \exists x A \iff \exists a \in D_w \, (w \Vdash A(a))$,
    \item $w \Vdash \Box A \iff \forall v \in W\,(w R v \Rightarrow v \Vdash A)$.
\end{itemize}
\end{definition}

We say that a formula $A$ is \emph{valid} in a model $\mathcal{M}$ if $w \Vdash A$ holds for all $w \in W$, and valid in a frame $\mathcal{F}$ if $A$ is valid in all models based on $\mathcal{F}$.

\begin{definition}[Generated submodel]
Let $\mathcal{M}=(W,R,\{D_w\}_{w\in W},\Vdash)$ be a Kripke model and let $w\in W$.
The \emph{generated submodel} of $\mathcal{M}$ by $w$ is the Kripke model $\mathcal{M}_w = (W_w, R {\upharpoonright} W_w, \{D_v\}_{v\in W_w}, \Vdash \!{\upharpoonright}W_w)$, where $W_w:=\{v\in W \mid wR^{*}v\}$ and $R^{*}$ is the reflexive transitive closure of $R$.
\end{definition}

\begin{lemma}[Generated submodel lemma]\label{gsl}
Let $\mathcal{M}=(W,R,\{D_w\}_{w\in W},\Vdash)$ be a Kripke model, let $w\in W$, and let $A$ be a modal sentence. 
Then $w \Vdash_{\mathcal M} A$ if and only if $w \Vdash_{\mathcal M_w} A$. 
\end{lemma}

A Kripke model $\mathcal{M}=(W,R,\{D_w\}_{w\in W},\Vdash)$ is said to be \emph{rooted} if there exists $r \in W$ such that $W = \{w \mid r R^{*} w\}$.
Such an element $r$ is called a \emph{root} of the model.

\begin{theorem}[cf.~Fitting and Mendelsohn~\cite{FM98}]
The quantified modal logic $\QK$ is Kripke complete.
That is, for any modal sentence $A$, the following are equivalent:
\begin{enumerate}
    \item $\QK \vdash A$. 
    \item $A$ is valid in all Kripke frames.
\end{enumerate}
\end{theorem}

Let $\QKbf$ be the logic obtained from $\QK$ by adding the universal closures of all instances of Barcan formula $\forall x \Box A(x) \to \Box \forall x A(x)$ as axioms.

\begin{theorem}[cf.~Fitting and Mendelsohn~\cite{FM98}]
The quantified modal logic $\QKbf$ is Kripke complete with respect to constant domain Kripke frames.
That is, for any modal sentence $A$, the following are equivalent:
\begin{enumerate}
    \item $\QKbf \vdash A$. 
    \item $A$ is valid in all constant domain Kripke frames.
\end{enumerate}
\end{theorem}

The following theorem is a consequence of one of Wolter and Zakharyaschev's results \cite{WZ01}. 

\begin{theorem}[Wolter and Zakharyaschev~{\cite[Corollary~5.7]{WZ01}}]\label{thm:WZ}
Let $A$ be a modal sentence in which only one variable appears. 
Then, the following are equivalent:
\begin{enumerate}
    \item $\QKbf \vdash A$. 
    \item $A$ is valid in all finite constant domain Kripke frames.
\end{enumerate}
\end{theorem}

A Kripke frame is said to be \emph{conversely well-founded} if it has no infinite $R$-increasing chains of elements.
As in the propositional case, a Kripke frame is transitive and conversely well-founded if and only if every modal sentence that is a theorem of $\QGL$ is valid in that frame.
On the other hand, the following fact is known.

\begin{theorem}[Montagna \cite{Mont84}]
The quantified modal logic $\QGL$ is Kripke incomplete.
That is, there exists a modal sentence which is valid in every
transitive and conversely well-founded Kripke frame,
but is not provable in $\QGL$.
\end{theorem}

\subsection{Provability predicates and derivability conditions}\label{ssec:pp}

Let $\LA = \{0,\su,+,\times,<\}$ be the language of first-order arithmetic.
For each $n \in \omega$, let $\num{n}$ denote the numeral
$\su^n(0)$ expressing $n$, which is the result of $n$ applications of $\su$ to $0$.
We fix a standard G\"odel numbering $\gn{\varphi}$ of $\LA$-formulas $\varphi$.

Throughout the present paper, let $T$ denote a consistent primitive recursive extension of Peano Arithmetic $\PA$ in the language $\LA$.
An $\LA$-formula $\PR_T(x)$ is called a \emph{provability predicate} of $T$ if for any $\LA$-sentence $\varphi$, $T \vdash \varphi$ if and only if $\PA \vdash \PR_T(\gn{\varphi})$.
Note that we do not assume that $\PR_T(x)$ is $\Sigma_1$ in general.
We fix a primitive recursive formula $\Proof_T(x, y)$ naturally expressing the relation ``$y$ is a $T$-proof of $x$''. 
Then, the $\Sigma_1$ formula $\Prov_T(x) : \equiv \exists y\, \Proof_T(x, y)$ is a provability predicate of $T$. 

In some situations, one requires that $\PA$ verifies that a provability predicate $\PR_T(x)$ satisfies certain natural properties of $T$-provability. 
For instance, in the usual proof of the second incompleteness theorem, it suffices that the following conditions hold.

\begin{definition}[Local derivability conditions]
\leavevmode
\begin{description}
    \item [$\DC{2}$] $\PA \vdash
    \PR_T(\gn{\varphi \to \psi})
    \to
    (\PR_T(\gn{\varphi}) \to \PR_T(\gn{\psi}))$.

    \item [$\DC{3}$] $\PA \vdash \PR_T(\gn{\varphi})
    \to
    \PR_T(\gn{\PR_T(\gn{\varphi})})$.
\end{description}
\end{definition}

\begin{theorem}[Formalized L\"ob's theorem]
Suppose that $\PR_T(x)$ satisfies $\DC{2}$ and $\DC{3}$.
Then, for any $\LA$-sentence $\varphi$, $\PA \vdash \PR_T(\gn{\PR_T(\gn{\varphi}) \to \varphi}) \to \PR_T(\gn{\varphi})$.
\end{theorem}

Moreover, the standard proof that a provability predicate satisfies $\DC{3}$ usually proceeds by establishing a stronger property, namely the formalized $\Sigma_1$-completeness $\SCU$ with parameters.
The derivability conditions with parameters are called the uniform derivability conditions.

\begin{definition}[Uniform derivability conditions]
\leavevmode
\begin{description}
    \item [$\DCU{1}$]
    if $T \vdash \forall \ora{x}\,\varphi(\ora{x})$, then $\PA \vdash \forall \ora{x}\,\PR_T(\gn{\varphi(\ora{\dot{x}})})$.
    
    \item [$\DCU{2}$] $\PA \vdash \forall \ora{x}\,
    \bigl(
    \PR_T(\gn{\varphi(\ora{\dot{x}})\to \psi(\ora{\dot{x}})})
    \to
    (\PR_T(\gn{\varphi(\ora{\dot{x}})}) \to \PR_T(\gn{\psi(\ora{\dot{x}})}))
    \bigr)$. 

    \item [$\DCU{3}$] $\PA \vdash \forall \ora{x}\,
    \bigl(
    \PR_T(\gn{\varphi(\ora{\dot{x}})})
    \to
    \PR_T(\gn{\PR_T(\gn{\varphi(\ora{\dot{x}})})})
    \bigr)$.
    
    \item [$\SCU$]
    if $\varphi(\ora{x})$ is a $\Sigma_1$-formula, then $\PA \vdash \forall \ora{x}\,
    \bigl(
    \varphi(\ora{x}) \to \PR_T(\gn{\varphi(\ora{\dot{x}})})
    \bigr)$.
\end{description}
Here, $\ora{x}$ is a finite tuple of variables. 
So, $\gn{\varphi(\ora{\dot{x}})}$ is an abbreviation for $\gn{\varphi(\dot{x}_0, \ldots, \dot{x}_{m-1})}$ that is a primitive recursive term corresponding to a primitive recursive function calculating the G\"odel number of $\gn{\varphi(\num{n_0}, \ldots, \num{n_{m-1}})}$ from $(n_0, \ldots, n_{m-1})$. 
\end{definition}

The strongest version of $\DC{2}$ is the following global one. 

\begin{definition}[Global derivability condition]
\leavevmode
\begin{description}
   \item [$\DCG{2}$] $\PA \vdash \forall x \forall y \bigl(\PR_T(x \dot{\to} y) \to (\PR_T(x) \to \PR_T(y)) \bigr)$.
\end{description}
Here $\dot{\to}$ is a primitive recursive term corresponding to a primitive recursive function calculating the G\"odel number of $\varphi \to \psi$ from those of $\varphi$ and $\psi$. 
\end{definition}

For a systematic study of provability predicates and various derivability conditions, see \cite{Kura20, Kura25}. 
For instance, the following theorem was presented in Buchholz's lecture note (cf.~\cite{Kura20,Kura25}). 

\begin{theorem}\label{Buc}
If $\PR_T(x)$ satisfies $\DCU{1}$ and $\DCU{2}$, then it satisfies $\SCU$.
\end{theorem}
Although Theorem \ref{Buc} shows that $\DCU{1}$ and $\DCU{2}$ imply 
$\SCU$, the assumption $\DCU{3}$ in the next theorem 
is not redundant, since a provability predicate to which the theorem is applied need not be $\Sigma_1$.
\begin{theorem}[Formalized L\"ob's theorem with parameters]\label{ULob}
Suppose that $\PR_T(x)$ satisfies $\DCU{1}$, $\DCU{2}$, and $\DCU{3}$. 
Then, for any formula $\varphi(\ora{x})$,
\[
\PA \vdash
\forall \ora{x}\Bigl(
\PR_T(\gn{\PR_T(\gn{\varphi(\ora{\dot{x}})}) \to \varphi(\ora{\dot{x}})})
\to
\PR_T(\gn{\varphi(\ora{\dot{x}})})
\Bigr).
\]
\end{theorem}

Let $\tau(v)$ be a formula weakly representing the theory $T$ in $\PA$, that is, for any $\LA$-sentence $\varphi$, $\varphi \in T$ if and only if $\PA \vdash \tau(\gn{\varphi})$. 
$\tau(v)$ is not necessarily the standard primitive recursive definition of $T$. 
Using $\tau$, we can uniformly construct a formula $\PR_\tau(x)$ saying that $x$ is provable in a theory defined by $\tau$ in a natural way. 
The formula $\PR_\tau(x)$ is called a \emph{Fefermanian provability predicate}. 
This notion was originally introduced by Feferman \cite{Fefe60}.
It is known that if $\tau(v)$ is $\Sigma_n$ for $n>0$, then
$\PR_\tau(x)$ is also a $\Sigma_n$ formula.

\begin{theorem}[Feferman \cite{Fefe60}]\label{Fef}
Every Fefermanian provability predicate satisfies $\DCU{1}$, $\DCG{2}$, and $\SCU$.
\end{theorem}

\subsection{Arithmetical interpretations}\label{ssec:ai}

To connect quantified modal logic with arithmetic, we introduce arithmetical interpretations.

\begin{definition}[Arithmetical interpretations]
Let $\PR_T(x)$ be a provability predicate of $T$.
An \emph{arithmetical interpretation} based on $\PR_T(x)$ is a map $f$ from modal formulas to $\LA$-formulas having the same free variables satisfying the following conditions: 
\begin{itemize}
    \item if $P$ is an $n$-ary predicate symbol and $f(P(x_0,\ldots,x_{n-1})) \equiv \varphi(x_0, \ldots, x_{n-1})$, then $f(P(y_0,\ldots, y_{n-1})) \equiv \varphi(y_0, \ldots, y_{n-1})$, 
    \item $f(A \circ B) \equiv f(A)\circ f(B)$ for $\circ \in \{\land,\lor,\to\}$,
    \item $f(\neg A) \equiv \neg f(A)$,
    \item $f(Q x A) \equiv Q x\, f(A)$ for $Q \in \{\forall, \exists\}$,
    \item $f(\Box A(x_0, \ldots, x_{n-1})) \equiv \PR_T(\gn{f(A)(\dot{x}_0, \ldots, \dot{x}_{n-1})})$.
\end{itemize}
\end{definition}

\begin{definition}
Let $\PR_T(x)$ be a provability predicate of $T$.
The quantified provability logic $\QPL(\PR_T)$ of $\PR_T(x)$ is the set of all modal sentences $A$ such that $T \vdash f(A)$ for all arithmetical interpretations $f$ based on $\PR_T(x)$.
\end{definition}

It is easily verified that if $\PR_T(x)$ satisfies $\DCU{1}$ and $\DCU{2}$, then $\QPL(\PR_T)$ contains $\QK$.
Indeed, $\DCU{1}$ yields the arithmetical soundness of necessitation, and $\DCU{2}$ yields the arithmetical soundness of the $\mathbf{K}$-axiom in the framework of quantified modal logic.
In particular, if $\PR_\tau(x)$ is a Fefermanian provability predicate, then by Theorem~\ref{Fef}, $\QPL(\PR_\tau)$ contains $\QK$.
On the other hand, if $\PR_T(x)$ satisfies $\DCU{1}$ and $\DCU{2}$, then by Theorem~\ref{Buc}, it also satisfies $\SCU$.
Therefore, if moreover $\PR_T(x)$ is a $\Sigma_1$ provability predicate, then it satisfies $\DCU{3}$.
Hence, by Theorem~\ref{ULob}, $\PR_T(x)$ satisfies the formalized L\"ob theorem with parameters.
In particular, $\QPL(\PR_T)$ contains $\QGL$.

Thus, for a Fefermanian provability predicate $\PR_\tau(x)$, if $\tau(x)$ is a $\Sigma_1$ formula, then $\QPL(\PR_\tau)$ contains $\QGL$.
However, in this case, Vardanyan \cite{Vard86} proved that the quantified provability logic of $\PR_\tau(x)$ is $\Pi^0_2$-complete.
In particular, it has strictly higher arithmetical complexity than $\QGL$ and hence cannot coincide with $\QGL$. 

On the other hand, Artemov and Japaridze proved the following significant arithmetical completeness result with respect to finite Kripke frames.

\begin{theorem}[Artemov and Japaridze \cite{AD90}]
Let $\tau(x)$ be a $\Sigma_1$ formula weakly representing $T$ in $\PA$.
For any modal sentence $A$, if $A \in \QPL(\PR_\tau)$, then $A$ is valid in all transitive and conversely well-founded finite Kripke frames.
\end{theorem}

They proved this theorem by constructing an appropriate arithmetical interpretation from each transitive conversely well-founded finite Kripke model by extending Solovay's method to the framework of quantified modal logic. 

The aim of the present paper is to extend this connection between arithmetical interpretations and finite Kripke frames which are not necessarily transitive nor conversely well-founded.
For this purpose, one must find provability predicates $\PR_T(x)$ for which $\QPL(\PR_T)$ does not necessarily contain $\QGL$.
There are essentially two possible ways to achieve this.
One is to drop the assumption that $\PR_T(x)$ is a $\Sigma_1$ provability predicate.
The other is to consider provability predicates which do not satisfy at least one of $\DCU{1}$ and $\DCU{2}$.

In Section~\ref{sec:main1}, we establish a result based on $\Sigma_2$ Fefermanian provability predicates.
In Section~\ref{sec:main2}, we establish a result based on $\Sigma_1$ provability predicates satisfying $\DCG{2}$ but not necessarily $\DCU{1}$.

\section{Arithmetical completeness with respect to $\Sigma_2$ Fefermanian provability predicates}\label{sec:main1}

In this section, we prove that for any conversely well-founded finite Kripke model of quantified modal logic, there exists a $\Sigma_2$ Fefermanian provability predicate $\PR_\tau(x)$ of $T$ and an arithmetical interpretation based on $\PR_\tau(x)$ such that the model can be embedded into arithmetic. 
As mentioned in the previous section, $\Sigma_2$ cannot be replaced by $\Sigma_1$ in general. 
Our proof is obtained by combining the proofs presented in Artemov and Japaridze \cite{AD90} and Kurahashi \cite{Kura18}. 
Therefore, we give only an outline of the proof. 

Let $\alpha(v)$ be a formula weakly representing $T$ in $\PA$.
We recursively define the sequence $\{\Con_\alpha^n\}_{n\in\omega}$ of $\LA$-sentences by $\Con_\alpha^0 :\equiv (0=0)$, and $\Con_\alpha^{n+1} :\equiv \neg \PR_\alpha(\gn{\neg \Con_\alpha^n})$.
If there exists $n\in\omega$ such that $T \vdash \neg \Con_\alpha^n$, then we say that the \emph{height} of $\alpha(v)$ is the least such $n$.
Otherwise, the height of $\alpha(v)$ is said to be $\infty$.
The following lemma is due to Beklemishev \cite{Bekl90} and a simplified proof of the lemma was presented in \cite{Kura18}.

\begin{lemma}[Cf.~{\cite[Lemma~4.3]{Kura18}}]\label{lem:height-infty}
There exists a $\Sigma_1$ formula $\alpha(v)$ weakly representing $T$ in $\PA$ such that the height of $\alpha(v)$ is $\infty$.
\end{lemma}

By Lemma~\ref{lem:height-infty}, we fix a $\Sigma_1$ formula $\alpha(v)$ weakly representing $T$ in $\PA$ whose height is $\infty$.

\begin{theorem}\label{main1}
    Suppose that a modal sentence $A$ is not valid in a conversely well-founded finite Kripke model. 
    Then, there exists a $\Sigma_2$ formula $\tau(x)$ weakly representing $T$ in $\PA$ and an arithmetical interpretation $f$ based on $\PR_\tau(x)$ such that $T \nvdash f(A)$. 
\end{theorem}

\begin{corollary}
    If a modal sentence $A$ is in 
    \[
        \bigcap \{\QPL(\PR_\tau) \mid \tau(v)\ \text{is a}\ \Sigma_2\ \text{formula weakly representing}\ T\ \text{in}\ \PA\},
    \]
    then $A$ is valid in all conversely well-founded finite Kripke frames. 
\end{corollary}

We present an outline of a proof of the theorem. 
Let $\mathcal{M}=(W,R,\{D_w\}_{w\in W},\Vdash)$ be a conversely well-founded finite rooted Kripke model. 
We may assume that $W = \{1, 2, \ldots, n\}$ and $1$ is a root of the model. 
We now adjoin a new point $0$ below $1$ and obtain a new Kripke model $\mathcal{M}^+ =(W^+, R^+, \{D_w\}_{w\in W^+}, \Vdash^+)$ where
\begin{itemize}
    \item $W^+ : = W\cup\{0\}$,
    \item $x R^+ y : \iff (x R y$ and $x,y\in W$) or ($x=0$ and $y=1$), 
    \item $D_0 : = D_1$, 
    \item for each $i \in W^+$, $\ora{k} \in D_i$, and predicate symbol $P$, \\
    $i \Vdash^+ P(\ora{k}) : \iff$ ($i \Vdash P(\ora{k})$ and $i \in W)$ or $(i = 0$ and $1 \Vdash P(\ora{k}))$. 
\end{itemize}

Let $(R^+)^{\mathrm{t}}$ be the transitive closure of $R^+$. 
Then $(R^+)^{\mathrm{t}}$ is transitive and conversely well-founded, and $0 (R^+)^{\mathrm{t}} i$ for every $i\in W$.
By the usual Solovay construction \cite{Solo76} for the transitive conversely well-founded finite Kripke frame $(W^+, (R^+)^{\mathrm{t}})$ of modal propositional logic, there exists a $\Sigma_2$ formula $\lambda(x)$ satisfying the following clauses: 
\begin{enumerate}
    \item $\PA \vdash \exists x\,\lambda(x)$,
    \item $\PA \vdash \forall x\forall y\,(\lambda(x)\land\lambda(y)\to x=y)$,
    \item if $i (R^+)^{\mathrm{t}} j$, then $\PA \vdash \lambda(\num{i}) \to \neg \PR_\alpha(\gn{\neg\lambda(\num{j})})$, 
    \item if $i\neq 0$, then $\PA \vdash \lambda(\num{i}) \to
    \PR_\alpha\Bigl(\gn{\bigvee_{i (R^+)^{\mathrm{t}} j}\lambda(\num{j})}\Bigr)$.
\end{enumerate}

Next, we define a $\Sigma_2$ formula $\gamma(v)$ by
\[
\gamma(v) :\equiv \exists y\exists z\,
\bigl(
\lambda(y)\land y\neq \num{0}\land y (R^+)^{\mathrm{t}} z \land \neg(y R^{+}z)
\land v=\gn{\neg\lambda(\dot{z})}
\bigr).
\]
Let $\tau(v) : \equiv \alpha(v)\lor\gamma(v)$.

The following lemma is proved by exactly simulating the proofs of Lemmas~4.7, 4.9, and 4.10 in \cite{Kura18} with replacing the binary relations $\prec$ and $\prec^\ast$ by $(R^+)^{\mathrm{t}}$ and $R^+$, respectively. 

\begin{lemma}[{\cite[Lemmas~4.7, 4.9, and 4.10]{Kura18}}]\label{lem:tau-basic}
Let $i,j\in W^{+}$.
\begin{enumerate}
    \item if $i\neq 0$, $i (R^+)^{\mathrm{t}} j$, and $\neg(iR^{+}j)$, then $\PA \vdash \lambda(\num{i}) \to \PR_\tau(\gn{\neg\lambda(\num{j})})$, 
    
    \item if $i\neq 0$, then $\PA \vdash \lambda(\num{i}) \to
    \PR_\tau\Bigl(\gn{\bigvee_{i R^{+}j}\lambda(\num{j})}\Bigr)$, 

    \item if $iR^{+}j$, then $\PA \vdash \lambda(\num{i}) \to \neg \PR_\tau(\gn{\neg\lambda(\num{j})})$, 

    \item $\mathbb{N} \models \lambda(\num{0})$, 

    \item for any $\LA$-sentence $\varphi$, $\varphi \in T$, $\mathbb{N} \models \tau(\gn{\varphi})$, and $\PA \vdash \tau(\gn{\varphi})$ are pairwise equivalent. 
\end{enumerate}
\end{lemma}

In particular, it follows from the last clause that $\PR_\tau(x)$ is actually a provability predicate of $T$. 

We now combine the above construction with the method of Artemov and Japaridze \cite{AD90}. 
For each $i\in W^{+}$, we assume $D_i\subseteq \omega$ and $0\in D_0$. 
Let $D := \bigcup_{i\in W^{+}} D_i$. 
By \cite{AD90}, for each $k \in D$ there exists a $\Sigma_1$ formula $\theta_k(x)$ satisfying the following lemma. 

\begin{lemma}[{\cite[Lemmas~2, 3, and 4]{AD90}}]\label{lem:AD}
Let $i \in W^+$ and $k, l \in D$. 
\begin{enumerate}
    \item If $k \in D_i$, then $\PA \vdash \lambda(\num{i}) \to \exists x\,\theta_k(x)$,
    \item $\PA \vdash \lambda(\num{i}) \to \forall x \bigl(\bigvee_{k_0 \in D} \theta_{k_0}(x) \bigr)$, 
    \item If $k \neq l$, then $\PA \vdash \forall x \neg \bigl(\theta_k(x) \land \theta_{l}(x) \bigr)$.
\end{enumerate}
\end{lemma}

We now define an arithmetical interpretation $f$ based on $\PR_\tau(x)$ as follows.
For each $m$-ary predicate symbol $P$, define
\[
f\bigl(P(x_0,\ldots,x_{m-1})\bigr)
: \equiv \bigvee_{\substack{i\in W\  \& \ k_0,\ldots, k_{m-1} \in D_i \\ i \Vdash^{+} P(k_0, \ldots, k_{m-1})}}
\bigl(\lambda(\num{i}) \land \theta_{k_0}(x_0) \land \cdots \land \theta_{k_{m-1}}(x_{m-1}) \bigr).
\]

 The following lemma is the analogue of Lemma 5 of \cite{AD90} in our setting.

\begin{lemma}[{\cite[Lemma 5]{AD90}}]\label{lem:truth}
Let $B(x_0, \ldots, x_{m-1})$ be any modal formula, $i\in W$, and $k_0, \ldots, k_{m-1} \in D_i$. 
\begin{enumerate}
    \item If $i \Vdash^{+} B(k_0, \ldots, k_{m-1})$, then 
    \[
        \PA \vdash \lambda(\num{i}) \land \theta_{k_0}(x_0) \land \cdots \land \theta_{k_{m-1}}(x_{m-1}) \to f(B(x_0, \ldots, x_{m-1})).
    \]

    \item If $i \nVdash^{+} B(k_0, \ldots, k_{m-1})$, then
    \[
        \PA \vdash \lambda(\num{i}) \land \theta_{k_0}(x_0) \land \cdots \land \theta_{k_{m-1}}(x_{m-1}) \to \neg f(B(x_0, \ldots, x_{m-1})).
    \]
\end{enumerate}
\end{lemma}

The proof proceeds by induction on the construction of $B$, exactly as in \cite{AD90}.
The quantifier cases follow from Lemma \ref{lem:AD}.
The modal case follows from Lemma~\ref{lem:tau-basic} together with the formalized $\Sigma_1$-completeness $\PA \vdash \theta_k(x) \to \PR_{\tau}(\gn{\theta_k(\dot{x})})$.

\begin{proof}[Proof of Theorem~\ref{main1}]
Suppose that a modal sentence $A$ is not valid in a conversely well-founded finite Kripke model.
By the generated submodel lemma (Lemma~\ref{gsl}), we may assume that $\mathcal{M}=(W,R,\{D_w\}_{w\in W},\Vdash)$ is a conversely well-founded finite rooted Kripke model and that $1 \in W$ is its root satisfying $1 \nVdash A$.

By the above construction, we obtain a formula $\lambda(x)$, a $\Sigma_2$ formula $\tau(v)$, and an arithmetical interpretation $f$ based on $\PR_\tau(x)$.
By Lemma~\ref{lem:truth}, we have $\PA \vdash \lambda(\num{1}) \to \neg f(A)$. 
Hence $\PA \vdash \neg \PR_\tau(\gn{\neg\lambda(\num{1})}) \to \neg \PR_\tau(\gn{f(A)})$.
Since $0R^{+}1$, it follows from Lemma~\ref{lem:tau-basic}.3 that $\PA \vdash \lambda(\num{0}) \to \neg \PR_\tau(\gn{\neg\lambda(\num{1})})$. 
Moreover, by the construction of $\lambda(x)$, we have $\mathbb{N} \models \lambda(\num{0})$.
Therefore, $\mathbb{N} \models \neg \PR_\tau(\gn{f(A)})$.
By Lemma \ref{lem:tau-basic}.5, this implies $T \nvdash f(A)$.
\end{proof}

Theorem~\ref{main1} applies to conversely well-founded finite Kripke
models. 
In its proof, this assumption is used in order to apply a Solovay-style construction to the transitive closure of the accessibility relation. 
It is natural to ask whether this restriction is essential.

\begin{problem}\label{prob:remove-cwf}
Can the conversely well-foundedness assumption in Theorem~\ref{main1}
be removed?
\end{problem}

\section{Arithmetical completeness with respect to $\Sigma_1$ provability predicates}\label{sec:main2}

In this section, we prove that for any finite constant domain Kripke model of quantified modal logic, there exists a $\Sigma_1$ provability predicate $\PR_T(x)$ of $T$ satisfying $\DCG{2}$ and an arithmetical interpretation based on $\PR_T(x)$ such that the model can be embedded into arithmetic.
The main difference from the result in Section~\ref{sec:main1} is that the provability predicate $\PR_T(x)$ obtained in this section is not Fefermanian.
As discussed in Subsection~\ref{ssec:ai}, the provability predicate constructed here does not necessarily satisfy $\DCU{1}$.
Another important point is that, unlike in Section~\ref{sec:main1}, we assume that Kripke models have constant domains rather than being conversely well-founded.

In \cite{Kura24}, a method was developed for embedding arbitrary finite Kripke models of propositional modal logic into arithmetic by using $\Sigma_1$ provability predicates.
However, this method cannot be directly extended to quantified modal logic.
The reason is that, when considering arithmetical interpretations of quantified modal logic, one must carefully deal with free variables and substitutions of numerals.
To carry out the argument of \cite{Kura24} in this setting, the provability predicate would be required to satisfy both $\DCU{1}$ and $\DCU{2}$.
However, the provability predicates we aim to construct cannot be guaranteed to satisfy both of these conditions as we mentioned in Subsection \ref{ssec:ai}.
Therefore, it is necessary to develop a different proof method.

Our strategy is as follows. 
We first construct a function $h$ corresponding to a Solovay function. 
The function $h$ starts with $h(0)=0$ and makes at most one transition to a nonzero value. 
We then define $\lambda(x)$ by $\exists y\,(h(y)=x)$.
In contrast to the construction of Artemov and Japaridze~\cite{AD90},
we can define the formulas $\theta_k(x)$ in a simpler way by using the
assumption that the given finite Kripke frame is constant domain.

A crucial point is that, although the definition of $h$ involves complicated conditions related to quantifiers, substitutions, and tautological consequence, the function $h$ is in fact primitive recursive. 
This fact plays a key role in our proof. 
Since the proof of this fact is somewhat lengthy, we postpone it to the Appendix.
As a consequence, $\lambda(x)$ is a $\Sigma_1$ formula.
With these preparations in hand, it remains to show that finite constant domain Kripke models can in fact be embedded into arithmetic.

Before proving the theorem, we prepare several notions and notation. 
An $\LA$-formula $\varphi$ is called \emph{propositionally atomic}
if either $\varphi$ is atomic or $\varphi$ is of the form $Qx\,\psi$, where $Q\in\{\forall,\exists\}$.
For each propositionally atomic formula $\varphi$, we fix a distinct propositional variable $p_\varphi$.
We define a primitive recursive injection $I$ from $\LA$-formulas to propositional formulas recursively as follows:
\begin{itemize}
    \item $I(\varphi)$ is $p_\varphi$ if $\varphi$ is propositionally atomic, 

    \item $I(\varphi\circ\psi)$ is $I(\varphi)\circ I(\psi)$ for $\circ\in\{\land, \lor, \to\}$, 
    
    \item $I(\neg\varphi)$ is $\neg I(\varphi)$.
\end{itemize}
Let $X$ be a finite set of $\LA$-formulas.
An $\LA$-formula $\varphi$ is called a \emph{tautological consequence} (\emph{t.c.}) of $X$ if $\bigwedge_{\psi\in X} I(\psi)\to I(\varphi)$ is a propositional tautology.
We write $X \vdash^{\mathrm{tc}} \varphi$ to mean that $\varphi$ is a t.c.~of $X$.

For each $n\in\omega$, let $F_n$ be the set of all $\LA$-formulas
whose G\"odel number is less than or equal to $n$.
We may assume that $F_0 = \emptyset$.
For each $n\in\omega$, let
\[
P_{T,n}:=\{\varphi \mid \mathbb{N} \models \exists y\le n\,\Proof_T(\ulcorner\varphi\urcorner,y)\},
\]
where $\Proof_T(x,y)$ is a standard primitive recursive proof predicate of $T$ introduced in Subsection \ref{ssec:pp}.
We may assume $P_{T,n}\subseteq F_n$.
If $P_{T,n}\vdash^{\mathrm{tc}} \varphi$, then $\varphi$ is provable in $T$.
These notions and facts can be formalized in $\PA$.

We are ready to state our main theorem of this section. 

\begin{theorem}\label{main2}
    Suppose that a modal sentence $A$ is not valid in a finite constant domain Kripke model. 
    Then, there exists a $\Sigma_1$ provability predicate $\PR_T(x)$ satisfying $\DCG{2}$ and an arithmetical interpretation $f$ based on $\PR_T(x)$ such that $T \nvdash f(A)$. 
\end{theorem}

\begin{corollary}
    If a modal sentence $A$ is in 
    \[
        \bigcap \{\QPL(\PR_T) \mid \PR_T(x)\ \text{is a}\ \Sigma_1\ \text{provability predicate of}\ T\ \text{satisfying}\ \DCG{2}\},
    \]
    then $A$ is valid in all finite constant domain Kripke frames. 
\end{corollary}

By combining this corollary to Theorem \ref{thm:WZ}, we obtain the following corollary. 

\begin{corollary}
    Let $A$ be a modal sentence in which only one variable appears. 
    If $A$ is in
    \[
        \bigcap \{\QPL(\PR_T) \mid \PR_T(x)\ \text{is a}\ \Sigma_1\ \text{provability predicate of}\ T\ \text{satisfying}\ \DCG{2}\},
    \]
    then $\QKbf \vdash A$. 
\end{corollary}

Let $\mathcal{M}=(W, R, D, \Vdash)$ be a rooted finite constant domain Kripke model.
We may assume that $W=\{1, 2, \ldots,n\}$ and $1$ is a root of $\mathcal{M}$. 
Moreover, we may assume that $D = \{0, 1, \ldots, d\}$. 
For each $k\in D$, we define a formula $\theta_k(x)$ as follows:
\begin{itemize}
    \item 
        If $k\neq 0$, then $\theta_k(x) : \equiv (x = \num{k})$, and 
    \item 
        $\theta_0(x) : \equiv (x = 0 \lor x > \num{d})$. 
\end{itemize}

The following lemma is easily proved and corresponds to Lemma \ref{lem:AD}. 

\begin{lemma}\label{theta}
Let $k, l \in D$. 
\begin{enumerate}
    \item $\PA \vdash \exists x \theta_k(x)$.
    \item If $k \neq l$, then $\PA \vdash \forall x \neg \bigl(\theta_k(x) \land \theta_{l}(x) \bigr)$. 
    \item $\PA \vdash \forall x \bigl(\bigvee_{k_0 \in D} \theta_{k_0}(x) \bigr)$.
    \end{enumerate}
\end{lemma}

For any tuple $\ora{a}$ of elements or variables, let $\mathrm{lh}(\ora{a})$ denote its length, and we write $a_{u}$ for its $u$-th component.
Using the Recursion Theorem, we recursively define a function $h$ as follows:
\begin{align*}
& h(0)=0, \\
& h(l+1)=
\begin{cases}
i & \text{if } h(l)=0 \ \&\  i=\min \{j \in W \mid j\ \text{is activated at}\ l\}, \\
h(l) & \text{otherwise},
\end{cases}
\end{align*}
where ``\emph{$j$ is activated at $l$}'' if there exist a number $s$, finite tuples $\ora{k_0},\ldots,\ora{k_s}$ of elements of $D$, $\LA$-formulas $\varphi(\ora{x_0}),\psi_1(\ora{x_1}),\ldots,\psi_s(\ora{x_s})$ in $F_l$, and finite tuples $\ora{b_0},\ora{b_1},\ldots,\ora{b_s}$ of numbers satisfying the following conditions: 
\begin{enumerate}
    \item for all $t$ with $1\leq t\leq s$, $\bigl[
\forall \ora{x_t}\,
(\lambda(\overline{j})\land \bigwedge_{u< \lh (\ora{k_t})}\theta_{k_{t,u}}(x_{t,u})
\to \psi_t(\ora{x_t}))
\in P_{T,l}\bigr]$, 

    \item $\forall \ora{x_0}\,
\bigl(\lambda(\overline{j})\land \bigwedge_{u<\lh(\ora{k_0})}\theta_{k_{0,u}}(x_{0,u})
\to \neg\varphi(\ora{x_0})\bigr)
 \in P_{T,l}$, 

    \item $\bigwedge_{t \leq s}  \bigwedge_{u< \lh (\ora{k_t})}\theta_{k_{t,u}}(\num{b_{t,u}})$ holds,  

    \item $P_{T,l} \cup\{\psi_1(\ora{\overline{b_1}}),\ldots,\psi_s(\ora{\overline{b_s}})\} \vdash^{\mathrm{tc}} \varphi(\ora{\overline{b_0}})$. 
    Equivalently, $P_{T,l} \vdash^{\mathrm{tc}} \bigwedge_{1 \leq t \leq s}\psi_t(\ora{\num{b_t}}) \to \varphi(\ora{\num{b_0}})$. 
\end{enumerate}
Here, $\lambda(x) :\equiv \exists y (h(y) = x)$.

At first sight, the binary relation ``$j$ is activated at stage $l$'' appears to be computably enumerable, since it contains unbounded existential quantifications over tuples $\ora{b_0},\ora{b_1},\ldots,\ora{b_s}$ of numbers.
However, this relation is in fact primitive recursive. 
As a consequence, we obtain the following proposition.

\begin{proposition}\label{PR_h}
The function $h$ is primitive recursive.
\end{proposition}

Although the proof of this proposition is essential in our proof of Theorem \ref{main2}, we postpone it to the Appendix in order not to interrupt the flow of the proof of Theorem~\ref{main2}.

Since $\lambda(x)$ is a $\Sigma_1$ formula, we have $\PA \vdash \lambda(x) \to \Prov_T(\gn{\lambda({\dot{x}})})$. 
Also, $\theta_k(x)$ is also a $\Sigma_1$ formula, $\PA \vdash \theta_k(x) \to \Prov_T(\gn{\theta_k({\dot{x}})})$.

\begin{lemma}\label{Proph2}
\leavevmode
\begin{enumerate}
\item 
$\PA \vdash \forall y \forall x (h(y) = x \land x \neq 0 \to \forall z \geq y \, h(z)=x)$.
\item
$\PA \vdash \forall x \forall y (0<x<y \land \lambda(x) \to \neg \lambda (y))$.
\end{enumerate}
\end{lemma}
\begin{proof}
Clause 1 is proved by using induction in $\PA$ and Clause 2 follows from clause 1.
\end{proof}

\begin{lemma}\label{Proph3}
\begin{enumerate}
    \item 
$\PA \vdash \exists x (\lambda(x) \land x \neq 0) \leftrightarrow \neg \Con_T$.
    \item 
For each $i \in W$, $T \nvdash \neg \lambda(\num{i})$.
    \item 
For each $n \in \omega$, $\PA \vdash \forall x \forall y (h(x)=0 \land h(x+1)=y \land y \neq 0 \to x>\num{n})$.
\end{enumerate}
\end{lemma}
\begin{proof}
1. Argue in $\PA$: 

$(\leftarrow)$: Suppose that $T$ is inconsistent. 
We find numbers $p$ and $j \neq 0$ such that $0=1 \in P_{T,p}$ and $\lambda(\num{j}) \to \neg 0=1 \in P_{T, p}$. 
Then $s=0$ and $\varphi \equiv 0=1$ witness the relation ``$j$ is activated at $p$'' to hold, and so we obtain $h(p+1) \neq 0$.

\medskip

$(\rightarrow)$: Suppose $\lambda(\num{i})$ and $i \neq 0$. Then there exists a number $l$ such that $h(l)=0$ and $h(l+1) =i$. 
Since $i$ is activated at $l$, we find $s$, $\ora{k_0}, \ldots,\ora{k_s} \in D$, $\varphi(\ora{x_0}), \psi_1(\ora{x_1}), \ldots, \psi_s(\ora{x_s})$ in $F_l$, and $\ora{b_0}, \ldots, \ora{b_s}$ satisfying the four conditions described above. 
It follows that $T$ proves
\[
\bigwedge_{1 \leq t \leq s} \forall \ora{x_t}
\left(\lambda(\overline{i})\land \bigwedge_{u< \lh (\ora{k_t})}\theta_{k_{t,u}}(x_{t,u})
\to \psi_t(\ora{x_t}) \right)
\]
and
\[
\forall \ora{x_0}\,
\left(\lambda(\overline{i})\land \bigwedge_{u<\lh(\ora{k_0})}\theta_{k_{0,u}}(x_{0,u})
\to \neg\varphi(\ora{x_0}) \right).
\]
Thus, $T$ proves the sentence
\[
\left(\lambda(\num{i}) \land \bigwedge_{t \leq s}  \bigwedge_{u< \lh (\ora{k_t})}\theta_{k_{t,u}}(\num{b_{t,u}}) \right) \to \left(\bigwedge_{1 \leq t \leq s}\psi_t(\ora{\num{b_t}}) \land \neg \varphi(\ora{\num{b_0}})\right). 
\]
Since $P_{T,l} \vdash^{\mathrm{tc}} \bigwedge_{1 \leq t \leq s}\psi_t(\ora{\num{b_t}}) \to \varphi(\ora{\num{b_0}})$, the formula $\bigwedge_{1 \leq t \leq s}\psi_t(\ora{\num{b_t}}) \to \varphi(\ora{\num{b_0}})$ is $T$-provable.
Hence, $T$ proves $\neg \bigl(\lambda(\num{i}) \land \bigwedge_{t \leq s}  \bigwedge_{u< \lh (\ora{k}_t)}\theta_{k_{t,u}}(\num{b_{t,u}}) \bigr)$.

On the other hand, since $\lambda(\num{i}) \land \bigwedge_{t \leq s}  \bigwedge_{u< \lh (\ora{k}_t)}\theta_{k_{t,u}}(\num{b_{t,u}})$ is a true $\Sigma_1$ sentence, it is provable in $T$ by formalized $\Sigma_1$-completeness. 
Therefore, $T$ is inconsistent.

\medskip

2. Suppose, towards a contradiction, that $T \vdash \neg \lambda(\num{i})$ for some $i \in W$. 
Equivalently, $\mathbb{N} \models \Prov_T(\gn{\neg \lambda(\num{i})})$.
We argue in the standard model $\mathbb{N}$: 
We find a number $p$ such that $\lambda(\num{i}) \to \neg \neg 0=1 \in P_{T,p}$ and $\neg 0=1 \in P_{T,p}$. 
Thus,  $s=0$ and $\varphi \equiv \neg 0=1$ witness the relation ``$i$ is activated at $p$'' to hold.
Then $h(p+1) \neq 0$, and it follows from Clause 1 that $\neg \Con_T$ holds.
This contradicts the consistency of $T$.

\medskip

3. Since $\mathbb{N} \models \Con_T$, it follows from Clause 1 that $\mathbb{N} \models \forall y \forall x (h(x)=y \to y=0)$. 
Thus, we obtain for any $n \in \omega$, $\mathbb{N} \models h(\num{n})=0$ and $\PA \vdash h(\num{n}) = 0$.
Therefore, $\PA \vdash h(x) = 0 \land h(x+1)=y \land x \leq \num{n} \to y=0$.
\end{proof}

We now define a provability predicate of $T$ satisfying the conditions required in Theorem \ref{main2}.
For this purpose, we construct a primitive recursive function $g$ which enumerates all theorems of $T$, and then define $\PR_g(x) : \equiv \exists y\,(g(y)=x)$.
The function $g$ is constructed step by step, and the construction consists of Procedures~1 and 2.
It starts with Procedure~1, where the values of $g$ are determined by referring to the proof predicate $\Proof_T(x,y)$.
At the same time, it refers to the values of the function $h$.
As long as $h$ keeps the value $0$, the construction of $g$ remains in Procedure~1.
Once $h$ takes a nonzero value, the construction of $g$ moves to Procedure~2.
In Procedure~2, formulas are output in a controlled way so that $\PR_g(x)$ satisfies $\DCG{2}$ while at the same time embedding the Kripke model $\mathcal{M}$ into arithmetic.

The method of defining a provability predicate by constructing such a function $g$ through Procedures~1 and 2 originates with Guaspari and Solovay \cite{GS79}.
The method of defining $g$ while referring to the behavior of a previously constructed function $h$ has been used in our recent work (see \cite{KK} for an overview of this line of research).

We start defining the function $g$. 
In the definition of $g$, we identify each formula with its G\"odel number.

\medskip

\textbf{Procedure 1.}

Stage $l$:
\begin{itemize}
\item If $h(l+1) =0$,
\begin{equation*}
  g(l)  = \begin{cases}
       \varphi & \text{if}\ l\ \text{is a}\ T\text{-proof of}\ \varphi,\ \text{that is,}\, \Proof_T(\varphi, l)\ \text{holds},\\
               0 & \text{otherwise}.
             \end{cases}
  \end{equation*}

Then, go to Stage $l+1$.

\item If $h(l+1) \neq 0$, go to Procedure 2.
\end{itemize}

\medskip

\textbf{Procedure 2.}

Suppose $l$ and $i \neq 0$ satisfy $h(l)=0$ and $h(l+1)=i$. 

Let 
$\{\xi_u \}_{u \in \omega}$ be a primitive recursive enumeration of all $\LA$-formulas in which each $\LA$-formula appears infinitely many times.

\begin{equation*}
g(l+u) : = \begin{cases} \xi_u & \text{ if } \xi_u \text{ is ready at}\ u,\\
  0 & \text{otherwise},
		\end{cases}
\end{equation*}
where a formula \emph{$\xi_u$ is ready at $u$} if there exist a number $s$, finite tuples $\ora{k_0},\ldots,\ora{k_s}$ of elements of $D$, $\LA$-formulas $\varphi_0(\ora{x_0}),\ldots,\varphi_s(\ora{x_s})$ in $F_{l-1}$, and finite tuples $\ora{a_0}, \ldots, \ora{a_s}$ of numbers satisfying the following conditions: 
\begin{enumerate}
    \item for all $j \in W$ with $i R j$ and for all $t \leq s$, $\forall \ora{x_t}\,
\bigl(\lambda(\overline{j})\land \bigwedge_{u< \lh (\ora{k_t})}\theta_{k_{t,u}}(x_{t,u})
\to \varphi_t(\ora{x_t}) \bigr)\in P_{T,l-1}$, 

    \item every element of $\ora{a_0}, \ldots, \ora{a_s}$ is $\leq u$,

    \item $\bigwedge_{t \leq s}  \bigwedge_{u< \lh (\ora{k_t})}\theta_{k_{t,u}}(\num{a_{t,u}})$ holds,

    \item $P_{T,l-1} \cup \{\varphi_0(\ora{a_0}), \ldots, \varphi_s(\ora{a_s}) \} \vdash^{\mathrm{tc}} \xi_u$.
\end{enumerate}

We have finished the definition of the function $g$. 

\begin{lemma}\label{lem_Prg_Prov}\label{lem:equiv}
$\PA + \Con_T \vdash \PR_g(x) \leftrightarrow \Prov_T(x)$.
\end{lemma}
\begin{proof}
We proceed in $\PA + \Con_T$: By Lemma \ref{Proph3}.1, we have $h(l) =0$ for all $l$, and the definition of $g$ continues to be in Procedure 1. Thus, 
 it follows that for any formula $\varphi$ and any number $p$, $g (p) = \varphi$ if and only if $p$ is a $T$-proof of $\varphi$. 
\end{proof}

Lemma \ref{lem:equiv} shows that $\PR_g(x)$ is actually a provability predicate of $T$. 
We prove that $\PR_g(x)$ satisfies $\DCG{2}$. 

\begin{lemma}\label{lemD2U}
$\PA \vdash \forall x \forall y \bigl(\PR_g(x \dot{\to} y) \to (\PR_g(x) \to \PR_g(y)) \bigr)$.

\end{lemma}
\begin{proof}
By Lemma \ref{lem_Prg_Prov}, we obtain $\PA + \Con_T \vdash \PR_g(x) \leftrightarrow \Prov_T(x)$. 
Since $\Prov_T(x)$ satisfies $\DCG{2}$, we have $\PA + \Con_T \vdash \forall x \forall y \bigl(\PR_g(x \dot{\to} y) \to (\PR_g(x) \to \PR_g(y)) \bigr)$.

We prove $\PA + \neg \Con_T \vdash \forall x \forall y \bigl(\PR_g(x \dot{\to} y) \to (\PR_g(x) \to \PR_g(y)) \bigr)$.
We argue in $\PA + \neg \Con_T$:
By Lemma \ref{Proph3}.1, there exist $l$ and $i \neq 0$ such that $h(l)=i$.
Suppose that formulas $\alpha_0 \to \alpha_1$ and $\alpha_0$ are output by $g$.
We only consider the case where $\alpha_0 \to \alpha_1$ and $\alpha_0$ are output by $g$ in Procedure 2, and other cases can be proved similarly.
Let $g(l+u_0)=\xi_{u_0} \equiv (\alpha_0 \to \alpha_1)$ and $g(l+u_1)= \xi_{u_1}\equiv \alpha_1$.
Then, for each $j \in \{0, 1\}$, $\alpha_j$ is ready at $u_j$. 

We find numbers $s_0$ and $s_1$, finite tuples $\ora{k_0},\ldots,\ora{k_{s_0}}$ and $\ora{l_0},\ldots,\ora{l_{s_1}}$ of elements of $D$, $\LA$-formulas $\varphi_0(\ora{x_0}),\ldots,\varphi_{s_0}(\ora{x_{s_0}}), \psi_0(\ora{y_0}),\ldots,\psi_{s_1}(\ora{y_{s_1}}) \in F_{l-1}$, and tuples $\ora{a_0}, \ldots, \ora{a_{s_0}}$ and $\ora{b_0}, \ldots, \ora{b_{s_1}}$ of numbers satisfying the conditions described above. 
We obtain $P_{T,l-1} \cup \{ \varphi_{0}(\ora{a_{0}}), \ldots, \varphi_{s_0}(\ora{a_{s_0}}) \} \vdash^{\mathrm{tc}} \alpha_0 \to \alpha_1$ and $P_{T,l-1} \cup \{ \psi_{0}(\ora{b_{0}}), \ldots, \psi_{s_1}(\ora{b_{s_1}}) \} \vdash^{\mathrm{tc}} \alpha_0$. 
It follows that
\[
P_{T,l-1} \cup \{ \varphi_{0}(\ora{a_{0}}), \ldots, \varphi_{s_0}(\ora{a_{s_0}}), \psi_{0}(\ora{b_{0}}), \ldots, \psi_{s_1}(\ora{b_{s_1}}) \} \vdash^{\mathrm{tc}} \alpha_1.
\]
Since $\alpha_1$ appears infinitely many times in the enumeration $\{\xi_u\}_{u \in \omega}$, there exists $u \geq \max \{u_0,u_1\}$ such that $\alpha_1 \equiv \xi_u$. 
These materials witness the relation ``$\xi_u$ is ready at $u$'' to hold, and so $g(l+u)= \xi_u \equiv\alpha_1$. 
Therefore, $\alpha_1$ is output by $g$. 

We go out from $\PA$. 
We conclude $\PA \vdash \forall x \forall y \left(\PR_g(x \dot{\to} y) \to (\PR_g(x) \to \PR_g(y)) \right)$ by the law of excluded middle.  
\end{proof}

Finally, we define an arithmetical interpretation $f$ based on $\PR_g(x)$ as follows: 
\[
    f\bigl(P(x_0,\ldots,x_{m-1})\bigr) : \equiv \bigvee_{\substack{i \in W \&\ k_0,\ldots,k_{m-1} \in D \\ \&\ i \Vdash P(k_0,\ldots,k_{m-1})}}(\lambda(\num{i}) \land \theta_{k_0}(x_0)\land \ldots\land \theta_{k_{m-1}}(x_{m-1})).
\]

\begin{lemma}\label{lem:embedding}
Let $i \in W$, $k_0, \ldots,k_{m-1} \in D$, and $B(x_0,\ldots,x_{m-1})$ be a modal formula.
\begin{enumerate}
    \item If $i \Vdash B(k_0,\ldots,k_{m-1})$, then $\PA \vdash \lambda(\num{i}) \land \theta_{k_0}(x_0) \land \cdots\land\theta_{k_{m-1}}(x_{m-1}) \to f(B(x_0,\ldots,x_{m-1}))$.
    \item
If $i \nVdash B(k_0,\ldots,k_{m-1})$, then $\PA \vdash \lambda(\num{i}) \land \theta_{k_0}(x_0) \land \cdots\land\theta_{k_{m-1}}(x_{m-1}) \to \neg f(B(x_0,\ldots,x_{m-1}))$.
\end{enumerate}
\end{lemma}
\begin{proof}
We prove Clauses 1 and 2 simultaneously by induction on the construction of $B$.
Suppose $B$ is an $m$-ary predicate symbol $P(x_0, \ldots,x_{m-1})$.

\medskip

1. If $i \Vdash P(k_0, \ldots,k_{m-1})$, then $\PA \vdash \lambda(\num{i}) \land  \theta_{k_0}(x_0) \land \cdots\land\theta_{k_{m-1}}(x_{m-1}) \to f(P(x_0,\ldots,x_{m-1}))$ follows immediately from the definition of $f$.

\medskip

2. Suppose $i \nVdash P(k_0,\ldots,k_{m-1})$. 
For $j \in W$ with $j \neq i$, $\PA \vdash \lambda(\num{i}) \to \neg \lambda(\num{j})$ follows from Lemma \ref{Proph2}.2.
By Lemma \ref{theta}.2, we obtain
\[
\PA \vdash \theta_{k_0}(x_0) \land \cdots \land \theta_{k_{m-1}}(x_{m-1}) \to \bigwedge_{\substack{l_0, \ldots,l_{m-1} \in D \\ i \Vdash P(l_0, \ldots ,l_{m-1})}} \neg (\theta_{l_0}(x_0) \land \cdots\land \theta_{l_{m-1}}(x_{m-1})).
\]
It follows that $\PA \vdash \lambda(\num{i}) \land \theta_{k_0}(x_0) \land \cdots \land \theta_{k_{m-1}}(x_{m-1}) \to \neg f(P(x_0, \ldots,x_{m-1}))$ by the definition of $f$.

\medskip

We prove the induction case. 
The cases $\neg, \vee, \land, \to$ are easily proved by using the induction hypothesis. 
We prove the cases $B (x_0, \ldots,x_{m-1}) \equiv \forall y \, C(y,x_0, \ldots,x_{m-1})$ and $B (x_0, \ldots,x_{m-1}) \equiv \Box C(x_0, \ldots,x_{m-1})$.

\medskip

Suppose $B$ is of the form $\forall y \, C(y,x_0, \ldots,x_{m-1})$.

\smallskip

1. Suppose $i \Vdash \forall y C(y,x_0,\ldots,x_{m-1})$. Then, for each $k \in D$, we obtain $i \Vdash C(k,k_0,\ldots,k_{m-1})$.
By the induction hypothesis, 
\[
    \PA \vdash \lambda(\num{i}) \land \theta_{k}(y) \land \theta_{k_0}(x_0) \land \cdots\land \theta_{k_{m-1}}(x_{m-1}) \to f(C(y,x_0, \ldots,x_{m-1})).
\]
Thus
\[
\PA \vdash \lambda(\num{i}) \land \left(\bigvee_{k \in D} \theta_{k}(y) \right) \land \theta_{k_0}(x_0) \land \cdots\land \theta_{k_{m-1}}(x_{m-1}) \to f(C(y,x_0, \ldots,x_{m-1}))
\]
and
\[
\PA \vdash \lambda(\num{i}) \land \left( \forall y \, \bigvee_{k \in D} \theta_{k}(y) \right) \land \theta_{k_0}(x_0) \land \cdots\land \theta_{k_{m-1}}(x_{m-1}) \to \forall y\, f(C(y,x_0, \ldots,x_{m-1})).
\]
It follows from Lemma \ref{theta}.3 that 
\[
\PA \vdash \lambda(\num{i}) \land \theta_{k_0}(x_0) \land \cdots\land \theta_{k_{m-1}}(x_{m-1}) \to f(\forall y\, C(y,x_0, \ldots,x_{m-1})).
\]

\medskip

2. Suppose $i \nVdash \forall y C(y,k_0,\ldots,k_{m-1})$. 
Then, there exists $k \in D$ such that
$i \nVdash C(k,k_0,\ldots,k_{m-1})$.
By the induction hypothesis, we obtain 
\[
\PA \vdash \lambda(\num{i}) \land \theta_k(y) \land \theta_{k_0}(x_0) \land \cdots\land \theta_{k_{m-1}}(x_{m-1}) \to \neg f(C(y,x_0,\ldots,x_{m-1})).
\]
It follows that
\[
\PA \vdash \lambda(\num{i}) \land \exists y \, \theta_k(y) \land \theta_{k_0}(x_0) \land \cdots\land \theta_{k_{m-1}}(x_{m-1}) \to \exists y \, \neg f(C(y,x_0,\ldots,x_{m-1})).
\]
By Lemma \ref{theta}.1, we conclude 
\[
\PA \vdash \lambda(\num{i})  \land \theta_{k_0}(x_0) \land \cdots\land \theta_{k_{m-1}}(x_{m-1}) \to \neg f(\forall y\, C(y,x_0,\ldots,x_{m-1})).
\]

\medskip

Suppose $B \equiv \Box C(x_0, \ldots,x_{m-1})$.

\smallskip

1. Suppose $i \Vdash \Box C(k_0,\ldots,k_{m-1})$.
Then, for each $j \in W$ with $i R j$, we obtain $j \Vdash C(k_0, \ldots,k_{m-1})$.
It follows from the induction hypothesis
that
\[
\PA \vdash \forall x_0 \cdots \forall x_{m-1}\bigl(\lambda(\num{j}) \land \theta_{k_0}(x_0) \land \cdots\land \theta_{k_{m-1}}(x_{m-1}) \to f(C(x_0, \ldots,x_{m-1}))\bigr).
\]
Then, there exists $p \in \omega$ such that for each $j \in W$ satisfying $iRj$, 
\begin{equation}\label{eq1}
    \forall x_0 \cdots \forall x_{m-1} \bigl(\lambda(\num{j}) \land \theta_{k_0}(x_0) \land \cdots \land \theta_{k_{m-1}}(x_{m-1}) \to f(C(x_0, \ldots x_{m-1})) \bigr) \in P_{T,p}. 
\end{equation}

We reason in $\PA + \lambda(\num{i})$: 
Let $a_0, \ldots, a_{m-1}$ be any numbers and suppose $\theta_{k_0}(\num{a_0}) \land \cdots \land \theta_{k_{m-1}}(\num{a_{m-1}})$. 
Let $l$ and $i \neq 0$ be such that $h(l)=0$ and $h(l+1) =i$. 
The construction of the function $g$ switches to Procedure 2 at Stage $l$.
We prove $f(C(\num{a_0}, \ldots, \num{a_{m-1}}))$ is output by $g$ in Procedure 2.

Let $u$ be any number such that $a_0, \ldots, a_{m-1} \leq u$ and $\xi_u \equiv f(C(\num{a_0}, \ldots,\num{a_{m-1}}))$. 
Then, $\xi_u$ is ready at $u$ because $s = 0$, $\ora{k}=(k_0, \ldots,k_{m-1})$, $f(C(x_0, \ldots,x_{m-1}))$, and $\ora{a} = (a_0, \ldots, a_{m-1})$ satisfy the required conditions. 
In fact, the first condition follows from (\ref{eq1}) and the fact $l > p$ follows from Lemma \ref{Proph3}.3. 
The second condition is the requirement that $a_0, \ldots, a_{m-1} \leq u$, which is already fulfilled. 
The third condition is also fulfilled because $\lambda(\num{i}) \land \theta_{k_0}(\num{a_0}) \land \cdots \land \theta_{k_{m-1}}(\num{a_{m-1}})$ holds. 
Finally, $P_{T, l-1} \cup \{f(C(\num{a_0}, \ldots,\num{a_{m-1}}))\} \vdash^{\mathrm{tc}} \xi_u$ trivially holds, which is the fourth condition. 
Therefore, we obtain $g(l+u) = \xi_u$. 
Thus $\PR_g(\gn{f(C(\num{a_0}, \ldots,\num{a_{m-1}}))})$ holds, that is, $f(\Box C(\num{a_0}, \ldots,\num{a_{m-1}}))$ holds.

\medskip

2. Suppose $i \nVdash \Box C(k_0, \ldots,k_{m-1})$. 
Then, there exists $j \in W$ such that $iRj$ and $j \nVdash C(k_0, \ldots,k_{m-1})$.
By the induction hypothesis, we obtain
\[
    \PA \vdash \forall x_0 \cdots \forall x_{m-1}\bigl(\lambda(\num{j}) \land \theta_{k_0}(x_0) \land \cdots \land \theta_{k_{m-1}}(x_{m-1}) \to \neg f(C(x_0, \ldots,x_{m-1})) \bigr)
\]
and there exists $p \in \omega$ such that this sentence is in $P_{T, p}$. 

We discuss in $\PA + \lambda(\num{i})$: 
Let $a_0, \ldots, a_{m-1}$ be any numbers and suppose $\theta_{k_0}(\num{a_0}) \land \cdots\land \theta_{k_{m-1}}(\num{a_{m-1}})$ holds. 
Let $\ora{k}=(k_0, \ldots,k_{m-1})$ and $\ora{a} = (a_0, \ldots,a_{m-1})$. 
Let $l$ and $i \neq 0$ be such that $h(l)=0$ and $h(l+1) = i$.

Suppose, towards a contradiction, that 
$f(C(\num{a_0}, \ldots,\num{a_{m-1}}))$ is output by $g$.
We distinguish the following two cases.

\medskip 

\noindent Case 1: $f(C(\num{a_0}, \ldots, \num{a_{m-1}}))$ is output in Procedure 1. \\
Since $l > p$ by Lemma \ref{Proph3}.3, we have that $f(C(\num{a_0}, \ldots, \num{a_{m-1}})) \in P_{T, l-1}$, and hence $P_{T, l-1} \vdash^{\mathrm{tc}} f(C(\num{a_0}, \ldots, \num{a_{m-1}}))$. 
Therefore, $s = 0$, $\ora{k} \in D$, $f(C(x_0, \ldots,x_{m-1})) \in F_{l-1}$, and $\ora{a}$ witness the relation ``$j$ is activated at $l-1$'' to hold. 
We obtain that $h(l) \neq 0$, this is a contradiction. 

\medskip

\noindent Case 2: $f(C(\num{a_0}, \ldots, \num{a_{m-1}}))$ is output in Procedure 2. \\
Suppose $g(l+u) = \xi_u \equiv f(C(\num{a_0}, \ldots, \num{a_{m-1}}))$.
Then, $\xi_u$ is ready at $u$, and so we find a number $s$, finite tuples $\ora{k_0}, \ldots,\ora{k_s}$ of elements of $D$, formulas $\varphi_0(\ora{x_0}),\ldots,\varphi_s(\ora{x_s}) \in F_{l-1}$, and finite tuples $\ora{b_0}, \ldots, \ora{b_s}$ of numbers satisfying the following four conditions: 
\begin{enumerate}
    \item for all $j' \in W$ with $i R j'$ and for all $t \leq s$, $\forall \ora{x_t}\,
\bigl(\lambda(\overline{j'})\land \bigwedge_{u< \lh (\ora{k_t})}\theta_{k_{t,u}}(x_{t,u})
\to \varphi_t(\ora{x_t}) \bigr)\in P_{T,l-1}$, 

    \item every element of $\ora{b_0}, \ldots, \ora{b_s}$ is $\leq u$,

    \item $\bigwedge_{t \leq s}  \bigwedge_{u< \lh (\ora{k_t})}\theta_{k_{t,u}}(\num{b_{t,u}})$ holds, 
    
    \item $P_{T,l-1} \cup \{\varphi_0(\ora{b_0}), \ldots, \varphi_s(\ora{b_s}) \} \vdash^{\mathrm{tc}} f(C(\num{a_0}, \ldots, \num{a_{m-1}}))$.
\end{enumerate}
Then, $s+1$, $\ora{k}, \ora{k_0}, \ldots,\ora{k_s} \in D$, $f(C(x_0, \ldots, x_{m-1})), \varphi_0(\ora{x_0}),\ldots,\varphi_s(\ora{x_s}) \in F_{l-1}$, and $\ora{a}, \ora{b_0}, \ldots, \ora{b_s}$ witness that the relation ``$j$ is activated at $l-1$'' to hold. 
It follows that $h(l)\neq 0$, a contradiction. 

We have proved that $\neg \PR_g(\gn{f(C(\num{a_0}, \ldots,\num{a_{m-1}}))})$ holds, that is, $\neg f(\Box C(\num{a_0}, \ldots,\num{a_{m-1}}))$ holds.
\end{proof}

We are ready to prove our main theorem of this section. 

\begin{proof}[Proof of Theorem~\ref{main2}]
Suppose that a modal sentence $A$ is not valid in a finite constant domain Kripke model.
By the generated submodel lemma (Lemma~\ref{gsl}), we may assume that $\mathcal{M}=(W,R, D, \Vdash)$ is a finite rooted constant domain Kripke model and that $1\in W$ is its root satisfying $1 \nVdash A$.
By the above construction, we obtain a $\Sigma_1$ provability predicate $\PR_g(x)$ satisfying $\DCG{2}$ and an arithmetical interpretation $f$ based on $\PR_g(x)$.
By Lemma \ref{lem:embedding}, we have $\PA \vdash \lambda(\num{1}) \to \neg f(A)$.
By Lemma~\ref{Proph3}.2, we have $T\nvdash \neg\lambda(\num{1})$.
It follows that $T\nvdash f(A)$.
\end{proof}

The constant domain assumption plays a substantial role in the present proof.
Since the provability predicate constructed in this section is not assumed to satisfy $\DCU{1}$, the existing methods of Artemov and
Japaridze~\cite{AD90} and of \cite{Kura24} cannot be applied directly. 
Therefore, we introduced a different mechanism such that the function $h$ is allowed to change its value according to the activation condition, which refers to the behavior of the function $g$.
This mechanism enables us to prove Lemma~\ref{lem:embedding} without relying on $\DCU{1}$.
In the expanding domain case, however, we do not know how to define the formulas $\theta_k(x)$ so as to fit the present construction.
For this reason, we restricted Theorem~\ref{main2} to constant domain Kripke models. 
We leave open whether this restriction can be removed.

\begin{problem}
Can the constant domain assumption in Theorem~\ref{main2} be removed?
\end{problem}

Another possible direction is to keep $\DCU{1}$ but drop $\DCU{2}$.
Indeed, $\Sigma_1$ provability predicates satisfying $\DCU{1}$ and $\DC{2}$ but not $\DCU{2}$ have been studied in \cite{Kura20,Kura21}. 
It is natural to ask whether such provability predicates can be used to embed finite Kripke models into arithmetic.

\begin{problem}
Can finite Kripke models of quantified modal logic be embedded into
arithmetic by means of $\Sigma_1$ provability predicates satisfying
$\DCU{1}$ and $\DC{2}$?
\end{problem}

\section*{Acknowledgements}
The first author was supported by JST SPRING, Grant Number JPMJSP2148. 
The second author was supported by JSPS KAKENHI Grant Number JP23K03200.

%
%
%

\appendix

\section*{Appendix: The primitive recursiveness of the function $h$}\label{app}
\addcontentsline{toc}{section}{Appendix}

In this appendix, we prove the primitive recursiveness of the function $h$ (Proposition \ref{PR_h}), whose proof was postponed in Section~\ref{sec:main2}.

We first show that whether two instances of terms or formulas obtained by substituting numerals are identical can be expressed by a formula in the language $\{0,\su\}$ of first-order arithmetic.
It is well known that the first-order theory of the structure $\mathbb{N}_{0, \su} = (\mathbb N,0,\su)$ admits effective quantifier elimination (cf.~Enderton \cite{Ende01}).
Moreover, inspecting the quantifier-elimination procedure, one has a primitive recursive procedure to determine the truth of a given sentence in this language.
Using this observation, we then prove that the binary relation
``$j$ is activated at stage $l$'' is in fact primitive recursive.
It follows that the function $h$ is primitive recursive.

Throughout the appendix, the expression $\alpha(u_0,\ldots,u_n)$ means that all (free) variables occurring in $\alpha$ are among $u_0,\ldots,u_n$.
By renaming variables if necessary, we may assume that the variables $u_0,\ldots,u_n$, $w_0,\ldots,w_m$, $\ora{v}$ are pairwise distinct, where $\ora{v}$ denotes a tuple of variables. 
Here $\equiv$ denotes syntactic identity of the resulting terms or formulas. 

\begin{lemma}\label{PR_equiv1}
Let $\ora{v}$ be any tuple of variables.
For any $\LA$-terms $t_0(u_0,\ldots,u_n,\ora{v})$ and $t_1(w_0,\ldots,w_m,\ora{v})$, there exists a formula $\varphi_{t_0, t_1, \ora{v}}(u_0,\ldots,u_n,w_0,\ldots,w_m)$ in the language $\{0, \su\}$ such that for all $a_0,\ldots,a_n,b_0,\ldots,b_m \in \omega$, the following are equivalent: 
\begin{enumerate}
    \item $t_0(\overline{a_0},\ldots,\overline{a_n},\ora{v})\equiv
t_1(\overline{b_0},\ldots,\overline{b_m},\ora{v})$, 
    \item $\mathbb{N}_{0,\su} \models
\varphi_{t_0, t_1, \ora{v}}(\overline{a_0},\ldots,\overline{a_n},\overline{b_0},\ldots,\overline{b_m})$.
\end{enumerate}
Furthermore, the formula $\varphi_{t_0, t_1, \ora{v}}$ can be found from $t_0$, $t_1$, and $\ora{v}$ in a primitive recursive way. 
\end{lemma}

\begin{proof}
We prove the lemma by induction on the number of occurrences of function symbols in $t_0$.

\medskip 

\noindent Base Step. 

\begin{description}
    \item [(i)] Suppose that $t_0(u_0,\ldots,u_n,\ora{v}) \equiv 0$.

    \begin{description}
        \item [(i-1)] If $t_1(w_0,\ldots,w_m,\ora{v}) \equiv 0$, then let $\varphi_{t_0, t_1, \ora{v}} : \equiv \top$.

        \item [(i-2)] If $t_1(w_0,\ldots,w_m,\ora{v})\equiv w_j$ for some $j \leq m$, then
\[
t_0(\overline{a_0},\ldots,\overline{a_n},\ora{v})\equiv
t_1(\overline{b_0},\ldots,\overline{b_m},\ora{v})
\iff
\overline{b_j} \equiv 0
\iff
\mathbb{N}_{0, \su} \models \overline{b_j} = 0.
\]
Let $\varphi_{t_0, t_1, \ora{v}} : \equiv (w_j = 0)$.

    \item [(i-3)] Otherwise, let $\varphi_{t_0, t_1, \ora{v}} : \equiv \bot$.

    \end{description}

    \item [(ii)] Suppose that $t_0(u_0,\ldots,u_n,\ora{v})\equiv u_i$ for some $i \leq n$.

    \begin{description}

        \item [(ii-1)] If $t_1(w_0,\ldots,w_m,\ora{v})\equiv \num{r}$, then let $\varphi_{t_0, t_1, \ora{v}} : \equiv (u_i= \num{r})$.

        \item [(ii-2)] If $t_1(w_0,\ldots,w_m,\ora{v})\equiv \su^r(w_j)$ for some $j \leq m$, then
\[
t_0(\overline{a_0},\ldots,\overline{a_n},\ora{v})\equiv
t_1(\overline{b_0},\ldots,\overline{b_m},\ora{v})
\iff
\overline{a_i}\equiv \su^r(\overline{b_j})
\iff
\mathbb{N}_{0, \su}\models \overline{a_i}=\su^r(\overline{b_j}).
\]
So let $\varphi_{t_0, t_1, \ora{v}} : \equiv (u_i=\su^r(w_j))$.

        \item [(ii-3)] Otherwise, let $\varphi_{t_0, t_1, \ora{v}} : \equiv \bot$.
    \end{description}

    \item [(iii)] Suppose that $t_0(u_0,\ldots,u_n,\ora{v})\equiv v$ for some $v \in \ora{v}$.

    \begin{description}
        \item [(iii-1)] If $t_1(w_0,\ldots,w_m,\ora{v})\equiv v$, then let $\varphi_{t_0, t_1, \ora{v}} : \equiv \top$.

        \item [(iii-2)] Otherwise, let $\varphi_{t_0, t_1, \ora{v}} : \equiv \bot$.

    \end{description}
\end{description}

\medskip

\noindent Induction Step. 

\smallskip

Suppose that the lemma holds for $\leq k$ and that the number of occurrences of function symbols in $t_0$ is $k+1$.

\begin{description}
    \item [(i)] Suppose that $t_0(u_0,\ldots,u_n,\ora{v})$ is $t_0'(u_0,\ldots,u_n,\ora{v})\ast t_0''(u_0,\ldots,u_n,\ora{v})$ for some $\LA$-terms $t_0'$ and $t_0''$, where $\ast\in\{+,\times\}$.

    \begin{description}
        \item[(i-1)] If $t_1(w_0,\ldots,w_m,\ora{v})$ is $t_1'(w_0,\ldots,w_m,\ora{v})\ast t_1''(w_0,\ldots,w_m,\ora{v})$ for some $\LA$-terms $t_1'$ and $t_1''$, then by the induction hypothesis, there exist formulas $\varphi_{t_0', t_1', \ora{v}}(u_0,\ldots,u_n,w_0,\ldots,w_m)$ and $\varphi_{t_0'', t_1'', \ora{v}}(u_0,\ldots,u_n,w_0,\ldots,w_m)$ in $\{0, \su\}$ such that
\[
t_0'(\ora{\overline{a}},\ora{v})\equiv
t_1'(\ora{\overline{b}},\ora{v})
\iff
\mathbb{N}_{0, \su}\models
\varphi_{t_0', t_1', \ora{v}}(\ora{\overline{a}},\ora{\overline{b}}),
\]
and 
\[
t_0''(\ora{\overline{a}},\ora{v})\equiv
t_1''(\ora{\overline{b}},\ora{v})
\iff
\mathbb{N}_{0, \su}\models
\varphi_{t_0'', t_1'', \ora{v}}(\ora{\overline{a}},\ora{\overline{b}}),
\]
Then, $t_0(\overline{a_0},\ldots,\overline{a_n},\ora{v})\equiv t_1(\overline{b_0},\ldots,\overline{b_m},\ora{v})$ is equivalent to
\[
    \mathbb{N}_{0, \su} \models \varphi_{t_0', t_1', \ora{v}}(\overline{a_0},\ldots,\overline{a_n},\overline{b_0},\ldots,\overline{b_m}) \land \varphi_{t_0'', t_1'', \ora{v}}(\overline{a_0},\ldots,\overline{a_n},\overline{b_0},\ldots,\overline{b_m}).
\]
Hence let $\varphi_{t_0, t_1, \ora{v}} : \equiv
\varphi_{t_0', t_1', \ora{v}}\land \varphi_{t_0'', t_1'', \ora{v}}$.

    \item [(i-2)] Otherwise, let $\varphi_{t_0, t_1, \ora{v}} :\equiv \bot$.

    \end{description}

    \item [(ii)] Suppose that $t_0(u_0,\ldots,u_n,\ora{v})\equiv \su\bigl(t_0'(u_0,\ldots,u_n,\ora{v})\bigr)$ for some $\LA$-term $t_0'$.

    \begin{description}
        \item [(ii-1)] If $t_1(w_0,\ldots,w_m,\ora{v})\equiv w_j$, then by the induction hypothesis, there exists a formula $\varphi_{t_0', t_1, \ora{v}}(u_0,\ldots,u_n,w_0,\ldots,w_m)$ in $\{0, \su\}$ such that
\[
t_0'(\overline{a_0},\ldots,\overline{a_n},\ora{v})\equiv \overline{c_j}
\iff
\mathbb{N}_{0, \su}\models
\varphi_{t_0', t_1, \ora{v}}(\overline{a_0},\ldots,\overline{a_n},\overline{b_0},\ldots, \overline{c_j}, \ldots, \overline{b_m}).
\]
Let $\varphi_{t_0, t_1, \ora{v}}(u_0,\ldots,u_n,w_0,\ldots,w_m)$ be the formula 
\[
\exists x\,
\bigl(\varphi_{t_0', t_1, \ora{v}}(u_0,\ldots,u_n,w_0,\ldots,x, \ldots, w_m)\land w_j = \su(x)\bigr).
\]
Then we obtain
\begin{align*}
& t_0(\overline{a_0},\ldots,\overline{a_n},\ora{v})\equiv
t_1(\overline{b_0},\ldots,\overline{b_m},\ora{v})\\
\iff &
\su\bigl(t_0'(\overline{a_0},\ldots,\overline{a_n},\ora{v})\bigr)\equiv \overline{b_j}\\
\iff & 
b_j > 0 \ \&\ t_0'(\overline{a_0},\ldots,\overline{a_n},\ora{v})\equiv \overline{b_j-1}\\
\iff &
b_j > 0 \ \&\ \mathbb{N}_{0, \su}\models \varphi_{t_0', t_1, \ora{v}}(\overline{a_0},\ldots,\overline{a_n},\overline{b_0},\ldots, \overline{b_{j}-1}, \ldots, \overline{b_m})\\
\iff &
\exists c \in \omega \, \text{s.t.}\ 
\bigl(\mathbb{N}_{0, \su}\models
\varphi_{t_0', t_1, \ora{v}}(\overline{a_0},\ldots,\overline{a_n},\overline{b_0},\ldots, \overline{c}, \ldots, \overline{b_m})
\land \overline{b_j}= \su(\overline{c})\bigr)\\
\iff &
\mathbb{N}_{0, \su}\models
\varphi_{t_0, t_1, \ora{v}}(\overline{a_0},\ldots,\overline{a_n},\overline{b_0},\ldots,\overline{b_m}).
\end{align*}

        \item [(ii-2)] If $t_1(w_0,\ldots,w_m,\ora{v})\equiv \su\bigl(t_1'(w_0,\ldots,w_m,\ora{v})\bigr)$ for some $\LA$-term $t_1'$, then by the induction hypothesis there exists a formula $\varphi_{t_0', t_1', \ora{v}}(u_0,\ldots,u_n,w_0,\ldots,w_m)$ such that
\[
t_0'(\overline{a_0},\ldots,\overline{a_n},\ora{v})\equiv
t_1'(\overline{b_0},\ldots,\overline{b_m},\ora{v})
\iff
\mathbb{N}_{0, \su}\models
\varphi_{t_0', t_1', \ora{v}}(\overline{a_0},\ldots,\overline{a_n},\overline{b_0},\ldots,\overline{b_m}).
\]
Then
\begin{align*}
& t_0(\overline{a_0},\ldots,\overline{a_n},\ora{v})\equiv
t_1(\overline{b_0},\ldots,\overline{b_m},\ora{v})\\
\iff &
\su\bigl(t_0'(\overline{a_0},\ldots,\overline{a_n},\ora{v})\bigr)\equiv
\su\bigl(t_1'(\overline{b_0},\ldots,\overline{b_m},\ora{v})\bigr)\\
\iff &
t_0'(\overline{a_0},\ldots,\overline{a_n},\ora{v})\equiv
t_1'(\overline{b_0},\ldots,\overline{b_m},\ora{v})\\
\iff &
\mathbb{N}_{0, \su}\models
\varphi_{t_0', t_1', \ora{v}}(\overline{a_0},\ldots,\overline{a_n},\overline{b_0},\ldots,\overline{b_m},\ora{v}).
\end{align*}
So let $\varphi_{t_0, t_1, \ora{v}} : \equiv \varphi_{t_0', t_1', \ora{v}}$.

        \item [(ii-3)] Otherwise, let $\varphi_{t_0, t_1, \ora{v}} : \equiv \bot$. \qedhere
    \end{description}
\end{description}
\end{proof}

\begin{lemma}\label{lem_equiv_fml}
Let $\ora{v}$ be any tuple of variables.
For any $\LA$-formulas $\alpha(u_0,\ldots,u_n,\ora{v})$ and $\beta(w_0,\ldots,w_m,\ora{v})$, there exists a formula $\psi_{\alpha, \beta, \ora{v}}(u_0,\ldots,u_n,w_0,\ldots,w_m)$ in $\{0, \su\}$ such that for all $a_0,\ldots,a_n,b_0,\ldots,b_m \in \omega$,
\[
\alpha(\overline{a_0},\ldots,\overline{a_n},\ora{v})\equiv
\beta(\overline{b_0},\ldots,\overline{b_m},\ora{v})
\iff
\mathbb{N}_{0, \su}\models
\psi_{\alpha, \beta, \ora{v}}(\overline{a_0},\ldots,\overline{a_n},\overline{b_0},\ldots,\overline{b_m}).
\]
Furthermore, the formula $\psi_{\alpha, \beta, \ora{v}}$ can be found from $\alpha$, $\beta$, and $\ora{v}$ in a primitive recursive way. 
\end{lemma}
\begin{proof}
We prove the lemma by induction on the construction of $\alpha$.

\begin{description}
    \item [(i)] Suppose that $\alpha$ is of the form $t_0(u_0,\ldots,u_n,\ora{v}) = t_1(u_0,\ldots,u_n,\ora{v})$ for some $\LA$-terms $t_0$ and $t_1$. 
    \begin{description}
    \item [(i-1)] If $\beta$ is of the form $s_0(w_0,\ldots,w_m,\ora{v}) = s_1(w_0,\ldots,w_m,\ora{v})$ for some $\LA$-terms $s_0$ and $s_1$, then by Lemma \ref{PR_equiv1}, for each $i\in\{0,1\}$, there exists a formula $\gamma_i(u_0,\ldots,u_n,w_0,\ldots,w_m)$ in $\{0, \su\}$ such that
\[
t_i(\overline{a_0},\ldots,\overline{a_n},\ora{v})\equiv
s_i(\overline{b_0},\ldots,\overline{b_m},\ora{v})
\iff
\mathbb{N}_{0, \su}\models
\gamma_i(\overline{a_0},\ldots,\overline{a_n},\overline{b_0},\ldots,\overline{b_m}).
\]
Hence $\alpha(\overline{a_0},\ldots,\overline{a_n},\ora{v})\equiv \beta(\overline{b_0},\ldots,\overline{b_m},\ora{v})$ is equivalent to
\[
    \mathbb{N}_{0, \su} \models \gamma_0(\overline{a_0},\ldots,\overline{a_n},\overline{b_0},\ldots,\overline{b_m}) \land \gamma_1(\overline{a_0},\ldots,\overline{a_n},\overline{b_0},\ldots,\overline{b_m}), 
\]
and so let $\psi_{\alpha, \beta, \ora{v}}$ be $\gamma_0(u_0,\ldots,u_n,w_0,\ldots,w_m)\land
\gamma_1(u_0,\ldots,u_n,w_0,\ldots,w_m)$.

    \item [(i-2)] Otherwise, let $\psi_{\alpha, \beta, \ora{v}} \equiv \bot$. 
\end{description}

    \item [(ii)] The case that $\alpha$ is of the form $t_0(u_0,\ldots,u_n,\ora{v}) < t_1(u_0,\ldots,u_n,\ora{v})$ for some $\LA$-terms $t_0$ and $t_1$ is similarly proved as in Case (i). 

    \item [(iii)] The cases where $\alpha$ is of the form built from a Boolean connective are similar: if $\alpha$ and $\beta$ have different principal connectives, then let $\psi_{\alpha, \beta, \ora{v}} \equiv\bot$; if they have the same principal connectives, then apply the induction hypothesis.

    \item [(iv)] Suppose that $\alpha(u_0,\ldots,u_n,\ora{v})\equiv \exists y\,\xi(u_0,\ldots,u_n,\ora{v},y)$. 
    The universal quantifier case is treated in the same way as the existential quantifier case.
    
    \begin{description}
        \item[(iv-1)] If $\beta(w_0,\ldots,w_m,\ora{v})\equiv \exists y\,\eta(w_0,\ldots,w_m,\ora{v},y)$, then by the induction hypothesis, there exists a formula $\psi_{\xi, \eta, \ora{v}, y}(u_0,\ldots,u_n,w_0,\ldots,w_m)$ such that $\xi(\overline{a_0},\ldots,\overline{a_n},\ora{v},y)\equiv
\eta(\overline{b_0},\ldots,\overline{b_m},\ora{v},y)$ is equivalent to 
\[
\mathbb{N}_{0, \su}\models
\psi_{\xi, \eta, \ora{v}, y}(\overline{a_0},\ldots,\overline{a_n},\overline{b_0},\ldots,\overline{b_m}).
\]
Hence $\exists y\,\xi(\overline{a_0},\ldots,\overline{a_n},\ora{v},y) \equiv
\exists y\,\eta(\overline{b_0},\ldots,\overline{b_m},\ora{v},y)$ is also equivalent to 
\[
\mathbb{N}_{0, \su}\models
\psi_{\xi, \eta, \ora{v}, y}(\overline{a_0},\ldots,\overline{a_n},\overline{b_0},\ldots,\overline{b_m}),
\]
and thus let $\psi_{\alpha, \beta, \ora{v}} \equiv \psi_{\xi, \eta, \ora{v}, y}$.

    \item [(iv-2)] Otherwise, let $\psi_{\alpha, \beta, \ora{v}} \equiv \bot$. \qedhere
\end{description}
\end{description}
\end{proof}

We are ready to prove the primitive recursiveness of the function $h$. 

\begin{proof}[Proof of Proposition \ref{PR_h}]
It suffices to show that the binary relation ``$j$ is activated at stage $l$'' is primitive recursive.

First observe that, except for the tuples of numbers $\ora{b_0},\ora{b_1},\ldots,\ora{b_s}$, all objects appearing in the definition of activation, except for the tuples $\ora{b_0},\ldots, \ora{b_s}$, range over finite sets whose sizes are bounded by
a primitive recursive function of $l$.
Indeed, the formulas $\varphi,\psi_1,\ldots,\psi_s$ are required to belong to the finite set $F_l$. 
Moreover, for each such formula, the tuples $\ora{k_t}$ range over the finite set $D^{<\omega}$ with length matching the number of free variables of the formula. 
Thus, the number of possible pairs $(\psi(\ora{x}),\ora{k})$ which can occur in the definition is effectively bounded in terms of $l$. 
So, in checking whether $j$ is activated at stage $l$, we may assume without loss of generality that no such pair is repeated among $(\varphi,\ora{k_0}), (\psi_1,\ora{k_1}),\ldots,(\psi_s,\ora{k_s})$. 
Hence the number $s$ is bounded by a primitive recursive function of $l$.

Therefore, for every fixed $s$ and $\varphi,\psi_1,\ldots,\psi_s$, it suffices to show that 
\[
\exists \ora{b_0} \exists \ora{b_1} \cdots \exists \ora{b_s}\, \left(\bigwedge_{t \leq s}  \bigwedge_{u< \lh (\ora{k_t})}\theta_{k_{t,u}}(\num{b_{t,u}})  
\ \&\ P_{T,l}\cup \{\psi_1(\ora{\overline{b_1}}),\ldots,\psi_s(\ora{\overline{b_s}})\}
\vdash^{\mathrm{tc}} \varphi(\ora{\overline{b_0}}) \right)
\]
can be decided by a primitive recursive procedure.
Our strategy is to convert this statement into an equivalent statement in the language $\{0, \su\}$ in a primitive recursive way. 
Then the primitive recursive procedure to determine the truth of a given sentence in the language $\{0, \su\}$ applies.
Since the statement $\bigwedge_{t \leq s}  \bigwedge_{u< \lh (\ora{k_t})}\theta_{k_{t,u}}(\num{b_{t,u}})$ is easily written in the language $\{0, \su\}$, we prove the conversion of the remaining part $P_{T,l}\cup \{\psi_1(\ora{\overline{b_1}}),\ldots,\psi_s(\ora{\overline{b_s}})\}
\vdash^{\mathrm{tc}} \varphi(\ora{\overline{b_0}})$. 

Let $X$ be the set of all propositionally atomic subformulas occurring in formulas in the set
\[
P_{T,l}\cup\{\varphi(\ora{x_0}),\psi_1(\ora{x_1}),\ldots,\psi_s(\ora{x_s})\}.
\]
Let $\sim$ be an equivalence relation on $X$. 
We assign to each equivalence class $C$ of $\sim$ a suitable propositional variable $p_C$. 
For each $\LA$-formula $\chi$ built from elements of $X$, we recursively define a propositional formula $\chi^\sim$ as follows:
\begin{itemize}
\item if $\chi\in X$ is propositionally atomic, then $\chi^\sim :\equiv p_{[\chi]}$, 
\item $(\chi\circ \xi)^\sim : \equiv \chi^\sim \circ \xi^\sim$ for $\circ\in\{\land,\lor,\to\}$, 
\item $(\neg \chi)^\sim : \equiv \neg(\chi^\sim)$. 
\end{itemize}

We define an equivalence relation $\sim^*$ on $X$ by
\[
 \alpha \sim^* \beta
 \ :\Longleftrightarrow\
\alpha(\ora{\num{b_0}}, \ldots, \ora{\num{b_s}})\equiv
\beta(\ora{\num{b_0}}, \ldots, \ora{\num{b_s}}).
\]
Let $I$ be the map assigning to an arithmetic formula the corresponding propositional
formula used in the usual definition of tautological consequence.
For each equivalence class $[\alpha]$ of $\sim^*$, we define the propositional
variable $p_{[\alpha]}$ by
\[
p_{[\alpha]} : \equiv I\bigl(\alpha(\ora{\num{b_0}}, \ldots, \ora{\num{b_s}})\bigr).
\]
This is well-defined by the definition of $\sim^*$.
Then, by induction on the construction of $\chi$, the following identity is verified: 
\begin{equation}\label{eq2}
\chi^{\sim^*} \equiv I\bigl(\chi(\ora{\overline{b_0}},\ldots,\ora{\overline{b_s}})\bigr).
\end{equation}

We say that an equivalence relation $\sim$ on $X$ is \emph{good} if the propositional formula
\[
\bigwedge \left(\{I(\rho) \mid \rho\in P_{T,l}\}\cup \{\psi_1(x_1)^\sim, \ldots, \psi_s(x_s)^\sim\}\right) \to \varphi(x_0)^\sim 
\]
is a tautology. 
In the light of (\ref{eq2}), the statement ``$P_{T,l} \cup \{\psi_1(\ora{\num{b_1}}), \ldots, \psi_s (\ora{\num{b_s}}) \} \vdash^{\mathrm{tc}} \varphi(\ora{\num{b_0}})$'' is equivalent to the goodness of $\sim^*$. 
So, we obtain the following claim. 

\begin{claim}
The following statements are equivalent:
\begin{enumerate}
\item 
    $P_{T,l} \cup \{\psi_1(\ora{\num{b_1}}), \ldots, \psi_s (\ora{\num{b_s}}) \} \vdash^{\mathrm{tc}} \varphi(\ora{\num{b_0}})$. 
\item
There exists a good equivalence relation $\sim$ on $X$ such that
\[
    \forall \alpha,\beta\in X\, (\alpha\sim\beta \iff \alpha\sim^* \beta) \bigr).
\]
\end{enumerate}
\end{claim}


There are only finitely many equivalence relations on $X$, and their number is effectively bounded in terms of $l$.
Moreover, for each such equivalence relation, whether it is good can be decided in a primitive recursive way.
Indeed, for a fixed equivalence relation $\sim$, the formula in question is a propositional formula whose size is bounded by a primitive recursive function of $l$.
Since propositional validity is primitive recursive, it follows that the property of being good is primitive recursive.
So, we can primitive recursively enumerate all good equivalence relations on $X$ as $\sim_0, \ldots, \sim_k$.

\begin{align*}
& P_{T,l} \cup \{\psi_1(\ora{\num{b_1}}), \ldots, \psi_s (\ora{\num{b_s}}) \} \vdash^{\mathrm{tc}} \varphi(\ora{\num{b_0}})\\
 \iff & \exists \sim
\bigl(
\text{$\sim$ is good }
\ \& \ 
\forall \alpha,\beta\in X\,
(\alpha\sim\beta \iff \alpha\sim^* \beta)
\bigr).\\
\iff & \bigvee_{0 \leq i \leq k}\Bigl(\bigwedge_{\substack{\alpha, \beta \in X \\ \alpha \sim_i \beta }} \alpha(\ora{\num{b_0}}, \ldots, \ora{\num{b_s}})\equiv
 \beta(\ora{\num{b_0}}, \ldots, \ora{\num{b_s}}) \ \& \ \bigwedge_{\substack{\alpha, \beta \in X \\ \alpha \not\sim_i \beta }} \alpha(\ora{\num{b_0}}, \ldots, \ora{\num{b_s}})\not \equiv
 \beta(\ora{\num{b_0}}, \ldots, \ora{\num{b_s}}) \Bigr) \\
\iff & \mathbb{N}_{0, \su} \models \bigvee_{0 \leq i \leq k}\Bigl(\bigwedge_{\substack{\alpha, \beta \in X \\ \alpha \sim_i \beta }} \psi_{\alpha,\beta, \emptyset}(\ora{\num{b_0}}, \ldots, \ora{\num{b_s}}) \land  \bigwedge_{\substack{\alpha, \beta \in X \\ \alpha \not\sim_i \beta }} \neg \psi_{\alpha,\beta, \emptyset}(\ora{\num{b_0}}, \ldots, \ora{\num{b_s}}) \Bigr).
\end{align*}

This completes the proof. 
\end{proof}

\end{document}